\documentclass[amsart,10pt]{article}

\usepackage[all]{xy}
\usepackage{enumitem}
\usepackage{amsmath}
\usepackage{tikz-cd}
\usepackage{mathrsfs}
\usepackage{amssymb}
\usepackage{amsbsy}
\usepackage{amsthm}
\usepackage{graphicx}
\usepackage{xypic}

\usepackage[textwidth=13cm,textheight=21cm]{geometry}

\usepackage{cite}
\usepackage{color}
\usepackage{yfonts}
\usepackage{multirow}
\usepackage{makeidx}
\usepackage[toc]{glossaries}
\usepackage{glossary-mcols}
\usepackage{chngcntr}
\usepackage{apptools}
\usepackage{hyperref}
\usepackage[toc,page]{appendix}

\AtAppendix{\counterwithin{exam}{chapter}}
\AtAppendix{\counterwithin{lem}{chapter}}
\AtAppendix{\counterwithin{prop}{chapter}}
\AtAppendix{\counterwithin{theo}{chapter}}
\AtAppendix{\counterwithin{definition}{chapter}}
\AtAppendix{\counterwithin{cortheo}{chapter}}
\AtAppendix{\counterwithin{cor}{chapter}}

\makeglossaries

\DeclareUnicodeCharacter{FE20}{\`}
\DeclareUnicodeCharacter{FE21}{\`}

\usepackage[utf8]{inputenc}

\usepackage[english]{babel}

\numberwithin{equation}{section}

\newtheorem{prop}{Proposition}[section]
\newtheorem{lem}{Lemma}[section]
\newtheorem{theo}{Theorem}[section]

\newtheorem{cor}{Corollary}[section]

\theoremstyle{definition}
\newtheorem*{claim}{Claim}
\newtheorem{exam}{Example}[section]
\newtheorem{definition}{Definition}[section]

\theoremstyle{remark}
\newtheorem{remark}{Remark}[section]

\title{An integral model of the perfectoid modular curve}
\date{}
\author{Juan Esteban Rodr\'iguez Camargo}

\newcommand{\bb}[1]{\mathbb{#1}}
\newcommand{\n}[1]{\mathcal{#1}}
\newcommand{\f}[1]{\mathfrak{#1}}
\newcommand{\s}[1]{\mathscr{#1}}

\DeclareMathOperator{\GL}{GL}
\DeclareMathOperator{\Tr}{Tr}
\DeclareMathOperator{\id}{id}
\DeclareMathOperator{\SL}{SL}

\DeclareMathOperator{\R}{\mathrm{R}}
\DeclareMathOperator{\D}{\mathrm{D}}
\DeclareMathOperator{\N}{N}

\DeclareMathOperator{\Hom}{Hom}

\DeclareMathOperator{\dotimes}{\otimes^{\mathrm{L}}}

\DeclareMathOperator{\coker}{coker}

\DeclareMathOperator{\Gal}{Gal}

\DeclareMathOperator{\Spf}{Spf}
\DeclareMathOperator{\Spec}{Spec}
\DeclareMathOperator{\Ell}{Ell}

\DeclareMathOperator{\Spa}{Spa}

\DeclareMathOperator{\Cusps}{Cusps}
\DeclareMathOperator{\Tate}{Tate}
\DeclareMathOperator{\HomSurj}{HomSurj}
\DeclareMathOperator{\Art}{Art}

\DeclareMathOperator{\Intperf}{\f{I}nt-\f{P}erf}
\DeclareMathOperator{\Perf}{\s{P}erf}

\newcommand{\mc}[1]{X(Np^{#1})}
\newcommand{\fmc}[1]{\mathfrak{X}(Np^{#1})}

\begin{document}
\maketitle

\begin{abstract}
We construct an  integral model of the perfectoid modular curve.  Studying this object, we prove some vanishing results for the coherent cohomology at perfectoid level.   We use a local duality theorem at finite level to compute duals for the coherent cohomology of the  integral perfectoid curve.  Specializing to the structural sheaf,  we can describe  the dual of the completed cohomology as the inverse limit of the integral cusp forms of weight $2$ and trace maps. 
\end{abstract}

\section*{Introduction}

Throughout this document we fix a prime number $p$, $\bb{C}_p$ the $p$-adic completion of an algebraic closure of $\bb{Q}_p$,  and $\{\zeta_{m}\}_{m\in \bb{N}}\subset \bb{C}_p$ a compatible system of primitive roots of unity.  Given a non-archimedean field $K$ we let $\n{O}_K$ denote its valuation ring.  We let $\overline{\bb{F}}_p$ be the residue field of $\n{O}_{\bb{C}_p}$  and $\breve{\bb{Z}}_p=W(\overline{\bb{F}}_p)\subset \bb{C}_p$ the ring of  Witt vectors.     Let $\bb{Z}_p^{cyc}$ and $\breve{\bb{Z}}_p^{cyc}$ denote the $p$-adic completions of the $p$-adic cyclotomic extensions of $\bb{Z}_p$ and $\breve{\bb{Z}}_p^{cyc}$ in $\bb{C}_p$ respectively.

Let $M\geq 1$ be an integer and $\Gamma(M)\subset \GL_2(\bb{Z})$ the principal congruence subgroup of level $M$. We fix $N\geq 3$ an integer prime to $p$.  For $n\geq 0$ we denote by $Y(Np^n)/ \Spec \bb{Z}_p$ the integral modular curve of level $\Gamma(Np^n)$ and $X(Np^n)$ its compactification,  cf. \cite{katz1985arithmetic}.   We denote by $\f{X}(Np^n)$ the completion of $X(Np^n)$ along its special fiber, and by $\n{X}(Np^n)$ its analytic generic fiber seen as an adic space over $\Spa(\bb{Q}_p, \bb{Z}_p)$,  cf. \cite{MR1734903}. 

In \cite{scholzeTorsCoho2015},  Scholze constructed the perfectoid modular curve of tame level $\Gamma(N)$. He proved that there exists a  perfectoid space $\n{X}(Np^\infty)$, unique up to a unique isomorphism, satisfying the tilde limit property 
\[
\n{X}(Np^\infty)\sim \varprojlim_n \n{X}(Np^n),
\]
see definition 2.4.1 of \cite{scholze2013moduli} and definition 2.4.2 of \cite{MR1734903}. 

The first result of this paper is the existence of a Katz-Mazur integral model of the perfectoid modular curve.  More precisely, we  prove the following theorem, see Section \ref{Proof Theorem} for the notion of a perfectoid formal scheme  
\begin{theo}
\label{Theo1Intro}
The inverse limit $\f{X}(Np^\infty)= \varprojlim_n \f{X}(Np^n)$ is  a formal perfectoid scheme over $\Spf \bb{Z}_p^{cyc}$ whose analytic generic fiber is naturally isomorphic to the perfectoid modular curve $\n{X}(Np^\infty)$.  
\end{theo}

 The integral perfectoid modular curve $\f{X}(Np^\infty)$ was previously constructed by Lurie in \cite{JacobFullLevel2019},  his method reduces the proof of perfectoidness to the ordinary locus via a mixed characteristic version of Kunz Theorem.  The strategy in this paper is more elementary: we use faithfully flat descent to deduce perfectoidness of $\f{X}(Np^\infty)$ from the  description of  the stalks at the $\overline{\bb{F}}_p$-points.    Then,  we deal with three different kind of points: 
 \begin{itemize}
 \item  The ordinary points where we use the Serre-Tate parameter to explicitly compute the deformation rings, cf.   \cite[\S 2]{katz1981serre}. 
 \item The cusps where we have explicit descriptions provided by the Tate curve,   cf.  \cite[\S  8-10]{katz1985arithmetic}.  
 \item The supersingular points where even though we do not compute explicitly the stalk,  one can proves that the Frobenius map is surjective modulo $p$.   
\end{itemize} 
  It worth to mention that the study of the ordinary locus in Lurie's approach and the one presented in this document are very related,   see  Proposition  2.2  of \cite{JacobFullLevel2019} and  Proposition \ref{theodefringord} down below.   
 
As an application of the integral model  we can prove vanishing results for the coherent cohomology of the  perfectoid modular curve.  Let $E^{sm}/X(N)$ be the semi-abelian scheme extending the universal elliptic curve over $Y(N)$,  cf. \cite{deligne1973schemas}. Let $e: X(N)\rightarrow E^{sm}$ be the unit section and $\omega_E= e^* \Omega^1_{E^{sm}/X(N)}$ the sheaf of invariant differentials.  For $n\geq 0$ we denote by $\omega_{E,n}$ the pullback of $\omega_E$ to $X(Np^n)$, and $D_n\subset X(Np^n)$ the reduced cusp divisor.  Let $k\in \bb{Z}$, we denote $\omega^{k}_{E,n}= \omega^{\otimes k}_{E,n}$ and $\omega^{k}_{E,n,cusp}= \omega^k_{E,n}(-D_n)$. 
 
 Let $\omega^k_{E,\infty}$ be the pullback of $\omega_E^k$ to $\f{X}(Np^\infty)$,  and $\omega^k_{E,\infty, cusp}$ the $p$-adic completion of the direct limit of the cuspidal modular sheaves $\omega^k_{E,n,cusp}$.   In the following we consider almost mathematics with respect to the maximal ideal of $\bb{Z}_p^{cyc}$. 
 
 \begin{theo}
 \label{Theo2Intro}
 Let $\n{F}= \omega^k_{E,\infty}$ or $\omega^k_{E,\infty,cusp}$  and  $\n{F}_{\mu}^+=\n{F}\widehat{\otimes}_{\s{O}_{\f{X}(Np^\infty)}} \s{O}^+_{\n{X}(Np^\infty)}$.   There is an almost quasi-isomorphism of complexes 
 \[
 \R \Gamma_{an}(\n{X}(Np^\infty), \n{F}^+_{\mu}) \simeq^{ae} \R \Gamma(\f{X}(Np^\infty),  \n{F}).
 \] 
 Moreover, the following holds
 \begin{enumerate}
 
\item The cohomology complex $\R \Gamma(\f{X}(Np^\infty), \n{F})$ is concentrated in degree $0$ if $k>0$, degree $[0,1]$ for $k=0$, and degree $1$ if $k<0$. 

\item For $k\in \bb{Z}$ and $i,s\geq 0$, we have $\mathrm{H}^i(\f{X}(Np^\infty),  \n{F})/p^s = \mathrm{H}^i(\f{X}(Np^\infty),  \n{F}/p^s)$ and \newline $\mathrm{H}^i(\f{X}(Np^\infty),  \n{F})= \varprojlim_s \mathrm{H}^i(\f{X}(Np^\infty),  \n{F}/p^s)$. 

\item  The cohomology groups $\mathrm{H}^i(\f{X}(Np^\infty),  \n{F})$ are torsion free. 
\end{enumerate}
 \end{theo}

Next,  we use Serre duality and Pontryagin duality to construct a local duality theorem for the modular curves at finite level.   In the limit one obtains the following theorem 
\begin{theo}
\label{Theo3Intro}
Let $\f{X}_n$ be the connected component of $\f{X}(Np^n)_{\breve{\bb{Z}}_p}$ defined as the locus where the Weil pairing of the universal basis of $E[N]$ is equal to $\zeta_{N}$.    We denote $\f{X}_{\infty}= \varprojlim_n \f{X}_n$.  Let $\n{F}= \omega^k_{E,\infty}$ or $\omega^k_{E,\infty, cusp}$ and   $\n{F}_n= \omega^k_{E,n}$ or $\omega^k_{E,n,cusp}$   respectively.  There is a natural $\GL_2(\bb{Q}_p)$-equivariant isomorphism 
\[
\Hom_{\breve{\bb{Z}}_p^{cyc} }(\mathrm{H}^i(\f{X}_{\infty},  \n{F}), \breve{\bb{Z}}_p^{cyc}) = \varprojlim_{n,  \widetilde{\Tr}_n} \mathrm{H}^{1-i}(X_n, \n{F}^\vee_n \otimes \omega^2_{E,n,cusp}), 
\]
where the transition maps in the RHS are given by normalized traces, and $\n{F}_n^\vee$ is the dual sheaf of $\n{F}_n$. 
\end{theo}
Finally, we specialize to the case $\n{F}= \s{O}_{\f{X}_{\infty}}$  where the completed cohomology appears.  Let $X_n$ be a connected component of $X(Np^n)_{\breve{\bb{Z}}_p}$ as in the previous theorem.  Let $i\geq 0$ and let  $\widetilde{\mathrm{H}}^i= \varprojlim_s \varinjlim_n \mathrm{H}^i_{et}(X_{n,\bb{C}_p},  \bb{Z}/p^s \bb{Z})$ be the completed $i$-th cohomology group,   where $X_{n,\bb{C}_p}= X_n\times_{\Spec \breve{\bb{Z}}_p[\zeta_{p^n}] }\Spec \bb{C}_p$.     Note that this is a slightly different version of Emerton's  completed cohomology \cite{MR2207783},  where one considers the \'etale cohomology with compact supports of $Y_{n,\bb{C}_p}\subset X_{n,\bb{C}_p}$.  Nevertheless, both cohomologies are related via the open and closed immersions $Y_n \subset X_n \supset D_n$.  Following the same ideas of \cite[\S 4.2]{scholzeTorsCoho2015} one can show that $\widetilde{\mathrm{H}}^i\widehat{\otimes}_{\bb{Z}_p} \n{O}_{\bb{C}_p}$ is almost equal to $\mathrm{H}^i_{an}(\n{X}_{\infty, \bb{C}_p},  \s{O}^+_{\n{X}_{\infty}})$,  in particular it vanishes for $i\geq 2$ and $\widetilde{\mathrm{H}}^0= \bb{Z}_p$.  Using the theorem above we obtain the following result
\begin{theo}
\label{Theo4Intro}
There is a $\GL_2(\bb{Q}_p)$-equivariant almost isomorphism of almost $\n{O}_{\bb{C}_p}$-modules 
\[
\Hom_{\n{O}_{\bb{C}_p}}( \widetilde{\mathrm{H}}^1\widehat{\otimes}_{\bb{Z}_p} \n{O}_{\bb{C}_p}, \n{O}_{\bb{C}_p})=^{ae} \varprojlim_{n, \widetilde{\Tr}_n} \mathrm{H}^0(X_{n,\bb{C}_p}, \omega^2_{E,\infty,cusp}). 
\]
\end{theo}

The outline of the paper is the following.  In Section \ref{RecallsModularCurves} we recall the construction of the integral modular curves at finite level;  they are defined as the moduli space parametrizing elliptic curves endowed with a Drinfeld basis of the torsion subgroups,  we will follow \cite{katz1985arithmetic}.  Then, we study the deformation rings of the modular curves at $\overline{\bb{F}}_p$-points.  For ordinary points we use the Serre-Tate parameter to describe the deformation ring at level $\Gamma(Np^n)$. We show that it represents   the moduli problem  parametrizing deformations of the $p$-divisible group $E[p^\infty]$, and a split of the connected-\'etale short exact sequence 
\[
0\rightarrow  \widehat{E} \rightarrow E[p^\infty] \rightarrow E[p^\infty]^{et}\rightarrow 0. 
\]
For cusps we refer to the explicit computations of \cite[\S 8 and 10]{katz1985arithmetic}. Finally, in the case of a supersingular point we prove that any element of the local deformation ring at level $\Gamma(Np^n)$ admits a $p$-th root   modulo $p$ at level $\Gamma(Np^{n+1})$. 
  
 In Section \ref{Proof Theorem}  we introduce the notion of a  perfectoid formal scheme.  We prove Theorem \ref{Theo1Intro} reducing to the formal deformation rings at $\overline{\bb{F}}_p$-points via faithfully flat descent.  We will say some words regarding Lurie's construction of $\f{X}(Np^\infty)$.  It is worth to mention that the tame level $\Gamma(N)$ is taken only for a more clean exposition,  by a result of Kedlaya-Liu about quotients of perfectoid spaces by finite group actions (Theorem 3.3.26 of  \cite{kedlaya2016relative}),  there are integral models of any tame level.  
 
 In Section \ref{Reminders Duality},  we use Serre and Pontryagin duality to define a local duality pairing for the coherent cohomology of vector bundles over an lci projective curve over a finite extension of $\bb{Z}_p$. 
 
 In Section \ref{COhomology of modular sheaves},   we compute the dualizing complexes of the modular curves at finite level.  We prove the cohomological vanishing of Theorem \ref{Theo2Intro} and its comparison with the cohomology of the perfectoid modular curve.    We prove the duality theorem at infinite level, Theorem \ref{Theo3Intro},  and specialize to $\n{F}=\s{O}_{\f{X}_{\infty}}$ to obtain  Theorem \ref{Theo4Intro}.

\textbf{Acknowledgments.} The construction of the integral  perfectoid modular curve initiated  as a problem of a \textit{m\'emoire} of M2  in 2019.   I should like to convey my special  gratitude to Vincent Pilloni for encouraging me to continue with the study of the coherent cohomology at perfectoid level,   and its application to  the completed cohomology.  I would like to thank George Boxer and Joaquin Rodrigues Jacinto  for  all the fruitful discussions of the subject.  I wish to express  special thanks to the anonymous referee  for the careful proofreading of  this paper, and for the numerous suggestions and corrections  that have notably improved the presentation of this text.    Particularly,   for  the remark  that Theorem \ref{Theo3Intro} holds for all the  sheaves involved and not only for the structural sheaf.     This work has been done while the author was a PhD student at the ENS de Lyon. 

\section{A brief introduction to the Katz-Mazur integral modular curves}

\label{RecallsModularCurves}

Let $N\geq 3$ be an integer prime to $p$ and $n\in \bb{N}$.   Let $\Gamma(Np^n)\subset \GL_2(\bb{Z})$ be  the principal congruence subgroup of level $Np^n$.  
\subsection*{Drinfeld bases}

We recall the definition of a Drinfeld basis for the $M$-torsion of an elliptic curve

\begin{definition}
 Let $M$ be a positive integer,   $S$   a scheme and $E$ an elliptic curve over $S$. A \textit{Drinfeld basis} of $E[M]$ is a morphism of  group schemes $ \psi:(\bb{Z}/M\bb{Z})^2\rightarrow E[M] $ such that  the following equality of effective divisors holds
 \begin{equation}
 \label{eqDrinfeldbasis}
 E[M]=\sum_{(a,b)\in (\bb{Z}/M\bb{Z})^2} \psi(a,b).
 \end{equation}
 We also write $(P,Q)=(\psi(1,0), \psi(0,1))$ for the Drinfeld basis $\psi$. 
\end{definition} 
 
 \begin{remark}
The left-hand-side of (\ref{eqDrinfeldbasis}) is  an effective divisor of $E/S$ being a finite flat group scheme over $S$. The right-hand-side is a sum of effective divisors given by the sections $\psi(a,b)$ of $S$ to $E$. Furthermore, if $M$ is invertible over $S$,   a homomorphism  $\psi: \bb{Z}/M\bb{Z}\rightarrow E[M]$ is a Drinfeld basis if and only if it is an isomorphism of group schemes, cf. \cite[Lem. 1.5.3]{katz1985arithmetic}.
 \end{remark}

\begin{prop}
 Let $E/S$ be an elliptic curve.   Let $(P,Q)$ be a Drinfeld basis of $E[M]$ and $e_M: E[M]\times E[M]\rightarrow \mu_M$ the Weil pairing. Then $e_M(P,Q)\in \mu_M^\times(S)$   is a primitive root of unity , i.e. a root of the $M$-th cyclotomic polynomial. 
\proof
\cite[Theo. 5.6.3]{katz1985arithmetic}. 
\endproof
\end{prop}

Let  $M\geq 3$.  From Theorem 5.1.1 and Scholie 4.7.0 of \cite{katz1985arithmetic},   the moduli problem parametrizing elliptic curves $E/S$ and  Drinfeld bases $(P,Q)$ of $E[M]$ is representable by an affine and  regular curve over $\bb{Z}$.  We denote this curve by $Y(M)$ and call it the \textit{(affine) integral modular curve of level $\Gamma(M)$}.  By an abuse of notation, we will  write $Y(M)$ for its scalar extension to $\bb{Z}_p$. 

The $j$-invariant is a finite flat morphism of $\bb{Z}_p$-schemes $j:Y(M)\longrightarrow \bb{A}^1_{\bb{Z}_p}$. The \textit{compactified  integral modular curve of level $\Gamma(M)$}, denoted by $X(M)$, is the normalization of $\bb{P}^1_{\bb{Z}_p}$ in $Y(M)$ via the $j$-invariant. The \textit{cusps} or the \textit{boundary divisor} $D$   is the closed reduced subscheme of $X(M)$  defined by $\frac{1}{j}=0$.  The curve $X(M)$  is projective over $\bb{Z}_p$ and a regular scheme.   
We  refer to $X(M)$ and $Y(M)$ simply as the modular curves of level $\Gamma(M)$.  

Let $E_{univ}/Y(M)$ be the universal elliptic curve and $(P_{univ,M},Q_{univ,M})$  the universal Drinfeld basis of $E_{univ}[M]$. Let $\Phi_{M}(X)$ be the $M$-th cyclotomic polynomial, and let $\bb{Z}_p[\mu^{\times}_{M}]$ denote the ring $\bb{Z}_p[X]/(\Phi_{M}(X))$.  The Weil pairing of $(P_{univ,M},Q_{univ,M})$ induces a morphism of $\bb{Z}_p$-schemes 
\[
e_M:Y(M)\rightarrow \Spec \bb{Z}_p[\mu_{M}^{\times}].
\] 

 The map $e_M$ extends uniquely to a map $e_M: X(M)\rightarrow \Spec \bb{Z}_p[\mu_{M}^\times]$ by normalization.   In addition,  $e_M$ is geometrically reduced, and has geometrically connected fibers.  
 
 Taking $N$ as in the beginning of the section,  and $n\in \bb{N}$ varying,  we construct  the commutative diagram 
\[
\begin{tikzcd}[column sep = small]
\cdots \ar[r] & \mc{n+1}\ar[d] \ar[r] & \mc{n} \ar[d]\ar[r] & \mc{n-1} \ar[r]\ar[d]&  \cdots \\
\cdots \ar[r] & \Spec(\bb{Z}_p[\mu^\times_{Np^{n+1}}]) \ar[r]& \Spec(\bb{Z}_p[\mu^\times_{Np^{n}}]) \ar[r]& \Spec(\bb{Z}_p[\mu^\times_{Np^{n-1}}]) \ar[r] & \cdots,
\end{tikzcd}
\]
  the upper horizontal arrows being  induced by the map \[(P_{univ,Np^{n+1}}, Q_{univ,Np^{n+1}})\mapsto (pP_{univ,Np^{n+1}}, pQ_{univ,Np^{n+1}})=(P_{univ,Np^{n}}, Q_{univ,Np^{n}}),  \] and the lower horizontal arrows by the natural inclusions. In fact,  the commutativity of the diagram  is a consequence of the compatibility of the Weil pairing with multiplication by $p$
\[
e_{Np^{n+1}}(pP_{univ,Np^{n+1}}, pQ_{univ,Np^{n+1}})= e_{Np^{n}}(P_{univ,Np^{n}}, Q_{univ,Np^{n}})^p
\]
cf.  Theorems 5.5.7 and 5.6.3 of \cite{katz1985arithmetic}.

\subsection*{Deformation rings at $\overline{\bb{F}}_p$-points }

Let $k=\overline{\bb{F}}_p$ be an algebraic closure of $\bb{F}_p$.  Let $\{\zeta_{Np^n}\}_{n\in \bb{N}}$ be a fixed  sequence of compatible primitive $Np^n$-th roots of unity, set $\zeta_{p^n}=\zeta_{Np^n}^{N}$.  Let $\breve{\bb{Z}}_{p}=W(k)$ denote the ring of integers of the $p$-adic completion of the maximal unramified extension of $\bb{Q}_p$.    In the next paragraphs we will study the deformation rings of the modular curve at the closed points  $X(Np^n)(k)$.   We let $X(Np^n)_{\breve{\bb{Z}}_p}$ denote the compactified modular curve over $\breve{\bb{Z}}_p$ of level $\Gamma(Np^n)$. Proposition 8.6.7 of \cite{katz1985arithmetic}   implies that  $\mc{n}_{\breve{\bb{Z}}_p}= \mc{n}\times_{\Spec \bb{Z}_p} \Spec \breve{\bb{Z}}_p$.

  There is an isomorphism  $\breve{\bb{Z}}_p [\mu_{Np^n}^{\times}]\cong \prod_{k\in (\bb{Z}/N\bb{Z})^{\times}} \breve{\bb{Z}}_p[\zeta_{p^n}]$ given by fixing a primitive $N$-th root of unity in $\breve{\bb{Z}}_p$.    Let $\mc{n}_{\breve{\bb{Z}}_p}^{\circ}$ be the connected component of the modular curve which corresponds to the root $\zeta_N$. In other words,  $\mc{n}_{\breve{\bb{Z}}_p}^\circ$ is the locus of $\mc{n}_{\breve{\bb{Z}}_p}$ where $e_{Np^n}(P_{univ,Np^n}, Q_{univ,Np^n})= \zeta_{Np^n}$.  We  denote  $P_{univ}^{(n)}:=N P_{univ,Np^{n}}$ and $Q_{univ}^{(n)}:=NQ_{univ,Np^n}$.  
  
   Finally,  given an elliptic curve $E/S$, we denote by $\widehat{E}$ the completion of $E$ along the identity section.

\vspace{10pt}

\subsubsection*{The ordinary points}

Let $\Art_k$ be the category of local artinian rings with residue field $k$, whose morphisms are the local ring  homomorphisms  compatible with the reduction to $k$.   Any object in $\Art_k$ admits an unique algebra structure over $\breve{\bb{Z}}_p$.  Let $\breve{\bb{Z}}_p[\zeta_{p^n}]\mbox{-}\Art_k$ denote the subcategory of $\Art_k$ of objects endowed with an algebra structure of $\breve{\bb{Z}}_p[\zeta_{p^n}]$ compatible with the reduction to $k$.   Following \cite{katz1981serre},   we use the  Serre-Tate parameter to describe the deformation rings at ordinary $k$-points of $X(Np^n)_{\breve{\bb{Z}}_p}$.

Let $E_0$ be an ordinary elliptic curve over $k$ and $R$  an object in $\Art_k$. A \textit{deformation of $E_0$ to $R$} is a pair $(E,\iota)$ consisting of an elliptic curve $E/R$ and an isomorphism $\iota:E\otimes_R k\rightarrow E_0$.  We define the deformation functor $\Ell_{E_0}: \Art_k \rightarrow \mbox{Sets}$ by the rule
\[
R \mapsto \{(E,\iota): \mbox{ deformation of $E_0$ to $R$}\}/\sim.
\]
Then $\Ell_{E_0}$ sends an artinian ring $R$ to the set of deformations of $E_0$ to $R$ modulo isomorphism.

Let $Q$ be a generator of the physical Tate module $T_pE_0(k)=T_p (E_0[p^\infty]^{et})$. Let $\bb{G}_m$ be the multiplicative group over $\breve{\bb{Z}}_p$ and $\widehat{\bb{G}}_m$ its formal completion along the identity.    We have the following pro-representability theorem 

\begin{theo}{\cite[Theo. 2.1]{katz1981serre}}
\label{SerreTate}
\begin{enumerate}

\item
 The Functor $\Ell_{E_0}$ is pro-representable by the formal scheme 
 \[\Hom_{\bb{Z}_p}(T_pE_0(k)\otimes T_p E_0(k),\widehat{\bb{G}}_m).\]
 The isomorphism is given by the Serre-Tate parameter $q$, which sends a deformation $E/R$ of $E_0$ to a bilinear form 
 \[
 q(E/R; \cdot,\cdot): T_pE_0(k)\times T_p E_0(k)\rightarrow \widehat{\bb{G}}_m(R).
 \]
By evaluating at the fixed generator $Q$ of $T_p E_0(k)$, we obtain the more explicit description 
\[
\Ell_{E_0}=\Spf( \breve{\bb{Z}}_p[[X]])
\]
 where $X=q(E_{univ}/\Ell_{E_0}; Q,Q)-1$. 

\item
 
 Let $E_0$ and $E_0'$ be ordinary elliptic curves over $k$, let $\pi_0: E_0\longrightarrow E_0'$ be a homomorphism and  $\pi_0^t: E_0' \longrightarrow E_0$  its dual.  Let $E$ and $E'$ be liftings of $E_0$ and $E_0'$ to $R$ respectively.  A necessary and sufficient condition for $\pi_0$  to lift to a homomorphism $\pi:E\longrightarrow E'$ is that 
\[
q(E/R;\alpha, \pi^t(\beta))=q(E'/R; \pi(\alpha), \beta)
\] 
for every $\alpha\in T_pE(k)$ and $\beta\in T_p E'(k)$.
 
\end{enumerate}
\end{theo}

  We deduce the following proposition describing the ordinary deformation rings of finite level: 

\begin{prop}
\label{theodefringord}
Let $x\in \mc{n}_{\breve{\bb{Z}}_p}^{\circ}(k)$ be  an ordinary   point,  say given by a triple $(E_0,P_0,Q_0)$,  and write $(P_0^{(n)},Q_0^{(n)})=(NP_0, NQ_0)$.   Let $A_{x}$ denote the deformation ring of $\mc{n}_{\breve{\bb{Z}}_p}^{\circ}$ at $x$. Then there is an isomorphism
\begin{equation}
\label{eqisoordinarypoint}
 A_{x}\cong \breve{\bb{Z}}_p[\zeta_{p^n}][[X]][T]/((1+T)^{p^n}-(1+X))=  \breve{\bb{Z}}_p[\zeta_{p^n}][[T]]
\end{equation}
such that: 

\begin{itemize}
\item[\normalfont{i.}]

the map (\ref{eqisoordinarypoint}) is $\breve{\bb{Z}}_p[\zeta_{p^n}]$-linear.

\item[\normalfont{ii.}]  the variable $1+X$ is equal to the Serre-Tate parameter $q(E_{univ}/A_{\overline{x}};Q,Q)$;

\item[\normalfont{iii.}] the variable $1+T$ is equal to the Serre-Tate parameter $q(E'_{univ}/A_{\overline{x}}, (\pi^{t})^{-1}(Q), (\pi^{t})^{-1}(Q))$ of the universal deformation $\pi: E_{univ}\rightarrow E'_{univ}$ of the \'etale isogeny $\pi_0: E_0\rightarrow E_0/C_0$, with $C_0=E_0[p^n]^{et}$.

\end{itemize}

\proof

The group scheme $E_0[N]$ is finite \'etale over $k$, which implies that   a deformation of $(E_0,P_0,Q_0)$ is equivalent to a deformation of $(E_0,P_0^{(n)},Q_0^{(n)})$.   The group $\SL_2(\bb{Z}/p^n\bb{Z})$ acts transitively on the set of Drinfeld bases of $E_0[p^n]$ with Weil pairing $\zeta_{p^n}$.  Without loss of generality,  we can assume that  $P_0^{(n)}=0$ and that $Q_0^{(n)}$ generates $E_0[p^n](k)$, see Theorem 5.5.2 of \cite{katz1985arithmetic}.  Let $E_{univ}$  denote the universal elliptic curve over $A_{x}$  and  $C\subset E_{univ}[p^n]$  the subgroup generated by $Q^{(n)}_{univ}$, it is an \'etale group lifting the \'etale group $C_0=E_0[p^n]^{et}$.  The base $(P_{univ}^{(n)},Q_{univ}^{(n)})$ provides  a splitting of the exact sequence 
\[
\begin{tikzcd}
0 \ar[r] & \widehat{E}_{univ} \ar[r] & E_{univ}[p^n] \ar[r, ] &  C_0 \ar[r] \ar[l, bend right, "Q_{univ}^{(n)}"'] & 0.
\end{tikzcd}
\]
Conversely,  let $R$ be  an object in $\breve{\bb{Z}}_p[\zeta_{p^n}]\mbox{-}\Art_k$    and $E/R$ a deformation of $E_0$.    Let  $C$   be an \'etale subgroup of $E[p^n]$ of rank $p^n$.     Then there exists a unique $Q^{(n)}\in C$ reducing to $Q_0^{(n)}$ modulo the maximal ideal. By Cartier duality,  there is   a unique $P^{(n)}\in \widehat{E}[p^n]$ such that $e(P^{(n)}, Q^{(n)})=\zeta_{p^n}$. The pair $(P^{(n)},Q^{(n)})$ is then a Drinfeld basis of $E[p^n]$ lifting $(P^{(n)}_0,Q^{(n)}_0)$ (cf.  Proposition 1.11.2 of \cite{katz1985arithmetic}).   We have  proved the equivalence of functors of  $\breve{\bb{Z}}_p[\zeta_{p^n}]\mbox{-}\Art_k$
\[
\left\{ \substack{ \mbox{Deformations $E$ of  $E_0$  and }   \\ \mbox{Drinfeld bases of $E[p^n]$} \\ \mbox{ with Weil pairing $\zeta_{p^n}$}  } \right\}\!/\!\sim\,  \longleftrightarrow   \left\{ \substack{ \mbox{Deformations  $E$ of  $E_0$  and }    \\ \mbox{\'etale subgroup  $C\subset E[p^n]$    of rank  $p^n$}   } \right\}\!/\! \sim.
\]
We also have a natural equivalence 
\[
 \left\{ \substack{ \mbox{Deformations  $E$ of  $E_0$  and }    \\ \mbox{\'etale subgroup $C\subset E[p^n]$     of rank  $p^n$}   } \right\}\! /\!\sim\,  \longleftrightarrow \left\{ \substack{ \mbox{Deformations of the \'etale isogeny}  \\ 	\mbox{$\pi_0: E_0\rightarrow  E_0/C_0$} }\right\}\! /\!\sim. 
\]
Let $E'_{univ}/\breve{\bb{Z}}_p[\zeta_{p^n}][[T]]$ denote the universal deformation of $E_0/C_0$. The universal \'etale point $Q_{univ}^{(n)}$ induces an \'etale isogeny of degree $p^n$  over $A_{x}$
\[
\pi: E_{univ}\rightarrow  E'_{univ}
\] 
lifting the  quotient $\pi_0:E_0\rightarrow E_0/C_0$. Furthermore, the dual morphism $\pi^t: E'_{univ}\rightarrow E_{univ}$ induces an isomorphism of the physical Tate modules $\pi^{t}: T_pE'_{univ}(k)\xrightarrow{\sim} T_p E_{univ}(k)$. Let $Q\in T_{p} E_{univ}(k)$ be the fixed generator, and   $Q'\in T_{p} E'_{univ}(k)$   its inverse under $\pi^t$.  Theorem \ref{SerreTate} implies 
\[
q(E_{univ}; Q,Q)=q(E_{univ}; Q,\pi^{t}(Q'))=q(E'_{univ};\pi(Q), Q' )=q(E'_{univ};Q',Q')^{p^n}. 
\]
We obtain the isomorphism 
\[
A_{x}\cong \breve{\bb{Z}}_p[\zeta_{p^n}][[X, T]]/((1+T)^{p^n}-(1+X))=\breve{\bb{Z}}_p[\zeta_{p^n}][[T]]
\]
where  $X=q(E_{univ};Q,Q)-1$ and $T=q(E'_{univ}; Q',Q')-1$.
\endproof
\end{prop}

\vspace{10pt}

\begin{remark}
Let $x\in X(Np^n)(k)$ be a  closed ordinary point. The special fiber of the map  $e_{Np^n}: X(Np^n)\rightarrow \Spec \bb{Z}_p[\mu_{Np^n}]$ is a union of Igusa curves with intersections at the supersingular points \cite[Theo. 13.10.3]{katz1985arithmetic}. The Igusa curves are smooth  over $\bb{F}_p$ \cite[Theo.  12.6.1]{katz1985arithmetic},  which implies that the deformation ring of $X(Np^n)$ at $x$ is isomorphic to a power series ring $\breve{\bb{Z}}_p[[T_n]]$ (cf.  discussion after Remark 3.4.4 of \cite{WeinsteinModinfinite}).  The content of the  previous proposition   is  the explicit  relation between the variables $T_n$ in the modular tower, see also Proposition 2.2 of \cite{JacobFullLevel2019}.
\end{remark}

\subsubsection*{The cusps}

Let $\Tate(q)/\bb{Z}_p((q))$ be the Tate curve, we recall from \!\cite[Ch. 8.8]{katz1985arithmetic}  that  it has $j$-invariant equal to \[1/q+744+\cdots.\] 
We consider the ring $\bb{Z}_p[[q]]$ as the completed stalk  of $\bb{P}^1_{\bb{Z}_p}$ at infinity. The Tate curve provides a description of the modular curve locally around the cusps, for that reason one can actually compute the formal deformation rings by means of this object,  see  \cite{katz1985arithmetic} and \cite{deligne1973schemas}. In fact, let $\widehat{\Cusps}[\Gamma(Np^n)]$ be the completion of the modular curve $\mc{n}_{\breve{\bb{Z}}_p}$ along the cusps. From the theory  developed in \cite[Ch. 8 and 10]{katz1985arithmetic}, more precisely Theorems 8.11.10 and  10.9.1, we deduce the following proposition:

\begin{prop}
\label{Theorem structure cusps 2}
We have an isomorphism of formal $\breve{\bb{Z}}_p[[q]]$-schemes 
\begin{equation*}
\widehat{\Cusps}(\left[\Gamma(Np^n)\right])  \xrightarrow{\sim} \displaystyle\bigsqcup_{\Lambda\in \HomSurj((\bb{Z}/Np^n\bb{Z})^2,\bb{Z}/Np^n\bb{Z})/\pm 1}  \Spf (\breve{\bb{Z}}_p[\zeta_{p^n}][[q^{1/Np^n}]]).
\end{equation*}
The morphism $\widehat{\Cusps}[\Gamma(Np^{n+1})]\rightarrow \widehat{\Cusps}[\Gamma(Np^n)]$ is induced  by the natural inclusion
\[
\breve{\bb{Z}}_p[\zeta_{p^n}][[q^{1/Np^{n}}]]\rightarrow \breve{\bb{Z}}_p[\zeta_{p^{n+1}}][[q^{1/Np^{n+1}}]]
\] 
on each respective connected component.
\end{prop}

\vspace{10pt}

\subsubsection*{The supersingular points}

Let $(x_n\in  \mc{n}_{\breve{\bb{Z}}_p})_{n\in \bb{N}}$ be a sequence of compatible supersingular points and $E_0$  the elliptic curve defined over $x_n$.  We denote by   $A_{x_n}$  the deformation ring of $X(Np^n)_{\breve{\bb{Z}}_p}$ at $x_n$.  Let $E_{univ}/A_{x_n}$ be the universal elliptic curve and $(P^{(n)}_{univ},Q^{(n)}_{univ})$ the  universal Drinfeld basis of  $E_{univ}[p^n]$. We fix a formal parameter  $T$ of $\widehat{E}_{univ}$.   Since $x_n$ is supersingular,   any $p$-power torsion point belongs to $\widehat{E}_{univ}$.  We will use the following lemma as departure point:

\begin{lem}{\cite[Theo. 5.3.2]{katz1985arithmetic}.} 
\label{Lemma generators max ideal}
The maximal ideal of the local ring $A_{x_n}$ is generated by $T(P^{(n)}_{univ})$ and $T(Q^{(n)}_{univ})$.

 \end{lem}

By  the Serre-Tate Theorem \cite[Theo. 1.2.1]{katz1981serre}, and the general moduli theory of $1$-dimensional formal groups over $k$ \cite{MR238854},   the deformation ring of $X(N)_{\breve{\bb{Z}}_p}$ at a   supersingular point is isomorphic to $\breve{\bb{Z}}_p[[X]]$. Moreover, the $p$-multiplication  modulo $p$ can be written as $[p](T)\equiv  V(T^p) \mod p$,  with $V\in k[[X]][[T]]$ the Verschiebung map $V: E_0^{(p)}\rightarrow E_0$.  Without loss of generality  we  assume that $V$ has the form 
\[
V(T)= XT+\cdots u(X)T^p+\cdots,
\]
with $V(T)\equiv T^p \mod X$. Using the Weierstrass Preparation Theorem we factorize $V(T)$ as 
\begin{equation}
\label{verschpowerseries}
V(T)=T(X+\cdots \tilde{u}(X)T^{p-1})(1+XT R(X,T)),
\end{equation} 
where $\tilde{u}(0)=1$ and $R\in k[[X,T]]$.  

\begin{prop}
\label{propSuperSingFrob}
The parameter $X$ is a $p$-power in  $A_{x_1}/p$. Moreover, the generators $T(P^{(n)}_{univ})$ and $T(Q^{(n)}_{univ})$ of the maximal ideal of $A_{x_n}$ are $p$-powers in  $A_{x_{n+1}}/p$.
\proof
The second claim follows from the first  and the equality $[p](T)\equiv V(T^p) \mod p$.  Consider $n=1$ and write $P=P^{(1)}_{univ}$ and $Q=Q^{(1)}_{univ}$. Let $F: E_0\rightarrow E_0^{(p)}$ and $V: E_0^{(p)}\rightarrow E_0$ denote the Frobenius and Verschiebung  homomorphisms respectively.  Using the action of $\GL_2(\bb{Z}/p\bb{Z})$,  we can assume that $P$ and $F(Q)$ are generators of $\ker F$ and  $\ker V$ respectively  (cf. Theorem 5.5.2 of \cite{katz1985arithmetic}).   We have the equality of divisors on $E_{univ}^{(p)}/(A_{x_n}/p)$
\begin{equation}
\label{kerversch}
\ker V= \sum_{i=0}^{p-1} [i\cdot F(Q)].
\end{equation}
The choice of the formal parameter $T$  gives  a formal parameter of $E^{(p)}_{univ}$ such that $T(F(Q))=T(Q)^p$. Therefore, from (\ref{kerversch}) we see that the roots of $V(T)/T$  are $\{[i]^{(p)}(T(Q)^p)\}_{1\leq i\leq p-1}$ where $[i]^{(p)}(T)$ is the $i$-multiplication of the formal group of  $E^{(p)}_{univ}$.  We obtain from (\ref{verschpowerseries})
\[
\frac{X}{\tilde{u}(X)}= (-1)^{p-1} \prod_{i=1}^{p-1} ([i]^{(p)}(T(Q)^p))= \left(\prod_{i=1}^{p-1}[i](T(Q))\right)^p,
\]
proving that $\frac{X}{\tilde{u}(X)}$ is a $p$-power in $A_{x_1}/p$. As $k[[X]]=k[[X/\tilde{u}(X)]]$ we are done. 
\endproof
\end{prop}

\begin{cor}
\label{corsupersingfrob}
The Frobenius $\varphi: \varinjlim_n A_{x_n}/p \rightarrow \varinjlim_n A_{x_n}/p$ is surjective. 

\proof
By induction on the graded pieces of the filtration defined by the ideal $(T(P^n),T(Q^n))$, one shows that $A_{x_n}/p$ is in the image of the Frobenius restricted to $A_{x_{n+1}}/p$.  
\endproof
\end{cor}

\begin{remark}
The completed local ring  at a geometric  supersingular point $\overline{x}$ of $X(Np^n)$  is difficult to describe.  For example,  its reduction modulo $p$ is the quotient of the power series ring  $k[[X,Y]]$ by some explicit principal ideal which is written in terms of the formal group law of $E$ at $\overline{x}$ \cite[Theo.  13.8.4]{katz1985arithmetic}. Weinstein gives in \cite{MR3529120} an explicit description of the deformation ring at a supersingular point of the modular curve at level $\Gamma(Np^\infty)$. In fact, Weinstein finds an explicit description of the deformation ring at infinite level of the Lubin-Tate space parametrizing  $1$-dimensional formal $\n{O}_K$-modules  of arbitrary height.  In particular,  he proves that the $\f{m}_{x_0}$-adic completion of the direct limit $\varinjlim_n A_{x_n}$ is a perfectoid ring.  The Corollary \ref{corsupersingfrob} says that the $p$-adic completion of $\varinjlim_{n} A_{x_n}$ is perfectoid, which is a slightly stronger result. 
\end{remark}

\section{Construction of the perfectoid integral model}
\label{Proof Theorem}

\subsection*{Perfectoid Formal  spaces}

In this section we introduce a notion of perfectoid formal  scheme which is already considered in \cite[Lemma 3.10]{MR3905467}, though not explicitly defined.      We start with the affine pieces 

\begin{definition}
An \textit{integral perfectoid ring} is a topological  ring $R$ containing a  non zero divisor $\pi$ such that  $p\in \pi^p R$, satisfying the following conditions:
\begin{itemize}

\item[i.] the ring $R$ is endowed with the $\pi$-adic topology. Moreover, it is  separated and complete. 

\item[ii.] the Frobenius morphism $\varphi: R/\pi R\rightarrow R/\pi^{p}R$ is an isomorphism. 
\end{itemize}
We call $\pi$ satisfying the previous conditions a \textit{pseudo-uniformizer} of $R$. 
\end{definition}

\begin{remark}
The previous definition of integral perfectoid rings   is well suited for $p$-adic completions of formal schemes. We do not consider the case where the underlying topology is not generated by a non-zero divisor, for example,   the ring $W(\overline{\bb{F}}_p)[[X^{1/p^\infty}, Y^{1/p^{\infty}}]]$ which is the $(p,X,Y)$-adic completion of the ring $W(\overline{\bb{F}}_p)[X^{1/p^{\infty}}, Y^{1/p^{\infty}}]$.   As is pointed out in Remark 3.8 of \cite{MR3905467},  the notion of being integral perfectoid does  not depend on the underlying  topology, however to  construct  a formal scheme it is necessary to fix one. 
\end{remark}

Let $R$ be an integral perfectoid ring with pseudo-uniformizer $\pi$, we attach to $R$ the formal scheme $\Spf R$ defined as the $\pi$-adic completion of $\Spec R$.    We say that $\Spf R$ is a  perfectoid formal affine scheme.    The following lemma says that the standard open subschemes of $\Spf R$ are perfectoid 

\begin{lem}
\label{lemLocalizationperf}
Let $f\in R$.  Then $R\langle f^{-1} \rangle= \varprojlim_n R/\pi^n [f^{-1}]$ is an integral perfectoid ring. 
\proof
Let $n,k\geq 0$,  as $\pi$ is not a zero divisor we have a short exact sequence 
\[
0\rightarrow R/\pi^n \xrightarrow{\pi^k} R/\pi^{n+k}\rightarrow R/\pi^k \rightarrow 0.
\]
Localizing at $f$ and taking inverse limits on $n$ we obtain 
\[
0\rightarrow R\langle f^{-1}\rangle\xrightarrow{\pi^k} R\langle f^{-1} \rangle\rightarrow R/\pi^{k} [f^{-1}]\rightarrow 0.
\]
Then $R\langle f^{-1} \rangle$ is $\pi$-adically complete and $\pi$ is not a zero-divisor.  On the other hand, localizing at $f$ the Frobenius map  
$\varphi: R/\pi \xrightarrow{\sim} R/\pi^p$ one gets 
\[
\varphi: R/\pi[f^{-1}]\xrightarrow{\sim} R/\pi^{p}[f^{-p}]= R/\pi^p [f^{-1}]
\]
which proves that $R\langle f^{-1}\rangle$ is an integral perfectoid ring. 
\endproof
\end{lem}

\begin{definition}
A \textit{perfectoid formal  scheme} $\f{X}$ is a formal scheme which admits an affine cover $\f{X}=\bigcup_i U_i$ by  perfectoid  formal affine schemes. 
\end{definition}

Let $F$ be equal to $\bb{Q}_p$  or $\bb{F}_p((t))$,  $\n{O}_F$ denote  the ring of integers of $F$ and $\varpi$ be a uniformizer of $\n{O}_F$.  Let  $\Intperf_{\n{O}_F}$ be the category of   perfectoid formal  schemes over $\n{O}_F$   whose structural morphism  is  adic, i.e. the category of formal  perfectoid  schemes $\f{X}/\Spf \n{O}_F$ such that $\varpi \s{O}_{\f{X}}$ is an ideal of definition of $\s{O}_{\f{X}}$.  Let $\Perf_F$ be the category of perfectoid spaces over $\Spa(F,\n{O}_F)$. 

\begin{prop} 
\label{PropgenericFiberFunctor}
 Let $R$ be an integral perfectoid ring and $\pi$ a pseudo-uniformizer.  The ring $R[\frac{1}{\pi}]$ is a perfectoid ring in the sense of Fontaine \cite{MR3087355}.  Furthermore,  there is a unique ``generic fiber'' functor 
\[
(-)_{\eta}: \Intperf_{\n{O}_F} \rightarrow Perf_F
\]
extending $\Spf R \leadsto \Spa(R[\frac{1}{\varpi}], R^+)$, where $R^+$ is the integral closure of $R$ in $R[\frac{1}{\varpi}]$.  Moreover, given $\f{X}$ a perfectoid formal scheme over $\n{O}_F$,  its generic fiber is universal for morphisms from perfectoid spaces to $\f{X}$.  Namely,  if $\n{Y}$ is a perfectoid space and $(\n{Y}, \s{O}^+_{\n{Y}})\to (\f{X}, \s{O}_{\f{X}})$ is a morphism of locally and topologically ringed spaces, then there is a unique map $\n{Y}\to \f{X}_{\mu}$ making the following diagram commutative  
\[
\begin{tikzcd}[row sep= 10pt, column sep =1pt]
(\n{Y}, \s{O}_{\n{Y}}^+)	\ar[rr]\ar[rd] & & (\f{X}_{\mu}, \s{O}_{\f{X}_{\mu}}^+)\ar[ld] \\ 
& (\f{X}, \s{O}_{\f{X}}) &
\end{tikzcd}
\]

\end{prop}

\begin{remark}
The universal property of the functor $(-)_{\mu}$  is Huber's characterization of the generic fiber of formal schemes in the case of perfectoid spaces,  see \cite[Prop.  4.1]{MR1306024}. 
\end{remark}

\begin{proof}
The first statement  is Lemma 3.21 of \cite{MR3905467}.   For the construction of the functor,  let $\f{X}$ be a perfectoid formal  scheme over $\n{O}_F$.   One can define $\f{X}_{\eta}$ to be the glueing of the affinoid spaces $\Spa(R[\frac{1}{\varpi}],, R^+)$ for $\Spf R\subset \f{X}$ an open  perfectoid formal afine subscheme,  this is well defined after Lemma \ref{lemLocalizationperf}.

We prove the universal property of the generic fiber functor.   Let $\n{Y}\in \Perf_K$ and let $f:(\n{Y}, \s{O}^+_{\n{Y}} )\rightarrow (\f{X}, \s{O}_{\f{X}})$ be  a morphism of locally and topologically ringed spaces.  First, if $\n{Y}=\Spa(S,S^+)$ is affinoid  perfectoid  and $\f{X}=\Spf R$ is perfectoid  formal  affine,   $f$ is determined by  the global sections map $f^*: R\rightarrow S^+$.  Then,  there exists a unique map of affinoid perfectoid rings $f^*_{\eta}: (R[\frac{1}{\varpi}], R^+)\rightarrow (S,S^+)$ extending $f^*$.  By  glueing  morphisms from  affinoid open subsets for a general $\n{Y}$,  one gets that $\f{X}_\eta:= \Spa(R[\frac{1}{\varpi}], R^+)$ satisfies the universal property.   For an arbitrary $\f{X}$,  one can glue the generic fibers of the   open perfectoid   formal   affine   subschemes  of $\f{X}$.
\end{proof}

We end this subsection with a theorem which reduces  the proof of the  perfectoidness of the integral modular curve  at any tame level  to the level $\Gamma(Np^\infty)$.

 \begin{theo}[Kedlaya-Liu]
 \label{theoinvperf}
 Let $A$ be a  perfectoid ring on which a finite group $G$ acts by continuous ring homomorphisms.  Then the invariant subring $A^G$ is  a perfectoid ring.  Moreover, if $R\subset A$ is an open integral perfectoid subring of $A$ then $R^G$ is an open integral perfectoid subgring of $A^G$. 
 \proof
 The first statement is Theorem  3.3.26 of \cite{kedlaya2016relative}.  The second statement follows from the description of open perfectoid subrings of $A$ as $p$-power closed subrings of $A^\circ$,  i.e open subrings of $A^\circ$ such that $x^p \in R$ implies $x\in R$, see Corollary 2.2 of  \cite{MorrowFoundPerfectoid}. 
 \endproof
 \end{theo}

\subsection*{The main construction}

Let $\n{X}(Np^\infty)$ denote  Scholze's perfectoid modular curve \cite{scholzeTorsCoho2015}.  Let $\bb{Z}_p^{cyc}$ be the $p$-adic completion of the $p$-adic cyclotomic integers $\varinjlim_n \bb{Z}_p[\mu_{p^n}]$.     Let $\f{X}(Np^n)$ be the completion of  $\mc{n}$ along its special fiber.  We have the following theorem

\begin{theo}
\label{theoexistperfint}
The inverse limit $\f{X}(Np^\infty):=\varprojlim_{n} \f{X}(Np^n)$ is a  $p$-adic perfectoid  formal  scheme,  it admits a structural map to  $\Spf \bb{Z}_p^{cyc}[\mu_{N}]$,  and its  generic fiber is naturally isomorphic to the perfectoid modular curve $\n{X}(Np^\infty)$.  Furthermore,  let $n\geq 0$, let  $\Spec R \subset X(Np^n)$ be an affine open subscheme, $\Spf \widehat{R}$  its $p$-adic completion and $\Spf \widehat{R}_{\infty}$ the inverse image in $\f{X}(Np^\infty)$.  Then $\widehat{R}_{\infty}= \left(\widehat{R}_{\infty}[\frac{1}{p}]\right)^{\circ}$ and 
\[
(\Spf \widehat{R}_{\infty})_{\eta}= \Spa(\widehat{R}_{\infty}[\frac{1}{p}], \widehat{R}_{\infty}). 
\]

\end{theo}

\begin{remark}
The previous result gives a different proof of Scholze's theorem that the generic fiber $\n{X}(Np^\infty)$ is a perfectoid space by more elementary means. 
\end{remark}

\begin{proof}
 The maps between  the (formal) modular curves are finite and flat.  Then  $\fmc{\infty}:=\varprojlim_{n} \fmc{n}$ is a flat $p$-adic   formal scheme over $\bb{Z}_p$.   Fix  $n_0\geq 0$,   let $\Spec R\subset \mc{n_0}$  and  $\Spf \widehat{R}\subset \fmc{n_0}$ be as in the theorem.   For $n\geq n_0$,  let $\Spec R_n$ (resp. $\Spf \widehat{R}_n$) denote the inverse image  of $\Spec R$ (resp. $\Spf \widehat{R}$) in $\Spec \mc{n}$ (resp. $\fmc{n}$).  Let $R_{\infty}:=\varinjlim_n R_n$ and let $\widehat{R}_{\infty}$ be its $p$-adic completion. 
  
\begin{claim}  $\widehat{R}_{\infty}$ is an integral perfectoid ring,   equal to  $(\widehat{R}_{\infty}[\frac{1}{p}])^{\circ}$. 
\end{claim} 
 
 Suppose that the claim holds,  it is left to show that $\f{X}(Np^\infty)_{\eta}$ is the perfectoid modular curve $\n{X}(Np^\infty)$.   There are natural maps of  locally and topologically  ringed spaces \[(\n{X}(Np^n),  \s{O}^+_{\n{X}(Np^n)})\rightarrow (\f{X}(Np^n), \s{O}_{\f{X}(Np^n)}).\] 
We have $\n{X}(Np^{\infty})\sim \varprojlim_n \n{X}(Np^n)$, where we use the notion of tilde limit  \cite[Def.  2.4.1]{scholze2013moduli}.   Then,  by $p$-adically completing  the inverse limit of the tower,  we obtain a map of locally and topologically  ringed spaces
\[
(\n{X}(Np^\infty), \s{O}^+_{\n{X}(Np^\infty)})\rightarrow ( \f{X}(Np^\infty),\s{O}_{\f{X}}(Np^\infty)). 
\]
 This provides a map $f:\n{X}(Np^\infty)\rightarrow (\f{X}(Np^\infty))_{\eta}$.  Since 
  \[
(\Spf \widehat{R}_{\infty})_{\eta}  =\Spa(\widehat{R}_\infty[1/p], \widehat{R}_{\infty})\sim \varprojlim_n \Spa(\widehat{R}_n[1/p], \widehat{R}_{n}),
 \]
 and the  tilde limit is unique in the category of perfectoid spaces \cite[Prop. 2.4.5]{scholze2013moduli},   the map $f$ is actually an isomorphism. 
 \end{proof}

\proof[Proof of the Claim] 
First,   by Lemma A.2.2.3 of \cite{Heuer2019PerfectoidGO} the ring $(\widehat{R}_n[\frac{1}{p}])^{\circ}$ is the  integral closure of $\widehat{R}_n$ in its generic fiber. By Lemma 5.1.2 of \cite{BhattLectures} and the fact that $R_n$ is a regular ring one gets that $\widehat{R}_n=(\widehat{R}_n[\frac{1}{p}])^{\circ}$ for all $n\geq n_0$.    As $R_{\infty}$ is faithfully flat over $R_n$ for all $n$,  one easily checks that $\widehat{R}_{\infty}\cap  \widehat{R}_n[\frac{1}{p}]=\widehat{R}_{n}$.  Moreover,  $R_{\infty}$ is integrally closed in its generic fiber, and by Lemma 5.1.2 of \textit{loc. cit.} again one obtains that  $\widehat{R}_{\infty}$ is integrally closed in $\widehat{R}_{\infty}[\frac{1}{p}]$.  Let $x\in \widehat{R}_n[\frac{1}{p}]$ be power bounded in $\widehat{R}_{\infty}[\frac{1}{p}]$,  then  $px^s\in  \widehat{R}_{\infty}$ for all $s\in \bb{N}$,  in particular $\{px^s\}_{s\in \bb{N}}\subset \widehat{R}_n$ which implies that $x\in \widehat{R}_n$.   This shows that $\varinjlim_{n} \widehat{R}_n$ is dense in $(\widehat{R}_{\infty}[\frac{1}{p}])^{\circ}$, taking $p$-adic completions one gets $\widehat{R}_{\infty}=(\widehat{R}_{\infty}[\frac{1}{p}])^{\circ}$.

 The Weil pairings evaluated at the universal Drinfeld basis $(P_{univ,Np^n},Q_{univ,Np^n})$  of $E[Np^n]$ induce compatible morphisms  $\f{X}(Np^n)\rightarrow \Spf \bb{Z}_p[\mu_{Np^n}]$. Taking inverse limits one gets the structural map $\f{X}(Np^\infty)\rightarrow \Spf  \bb{Z}_p^{cyc}[\mu_N]$.  In particular,  there exists  $\pi\in \widehat{R}_{\infty}$ such that $\pi^p=pa$ with $a\in \bb{Z}_p^{cyc,\times}$.    To prove that $\widehat{R}_{\infty}$ is integral perfectoid we need to show that the absolute  Frobenius map 
\[
\varphi: R_{\infty}/\pi  \longrightarrow R_{\infty}/p 
\]
is an isomorphism. The strategy is to prove this fact for the completed local rings of the stalks  of  $\Spec R_{\infty}/p$  and use faithfully flat descent.    

Injectivity is easy,  it follows from the fact that  $R_{\infty}$ is integrally closed in $R_{\infty}[1/p]$.    To show that  $\varphi$ is surjective,  it is enough to prove that the absolute Frobenius is surjective after a profinite \'etale base change.  Indeed, the relative Frobenius is an isomorphism for profinite \'etale base changes.   Let $S= R\otimes_{\bb{Z}_p} \breve{\bb{Z}}_p$ and let $\widehat{S}= \widehat{R}\widehat{\otimes}_{\bb{Z}_p} \breve{\bb{Z}}_p$ be the $p$-adic completion of $S$. We use similar notation for $S_n= R_n\otimes_{\bb{Z}_p}  \breve{\bb{Z}}_p$, $\widehat{S}_n$, $S_\infty$ and $\widehat{S}_{\infty}$.  We have to show that the absolute Frobenius 
 \begin{equation*}
 \varphi: S_{\infty}/\pi\rightarrow S_{\infty}/p
 \end{equation*}
 is surjective. 
 
  Let $x=(x_{n_0}, x_{n_0+1}, \cdots, x_n ,\cdots )$ be a $\overline{\bb{F}}_p$-point of $\Spf \widehat{S}_{\infty}$ which is an inverse limit of $\overline{\bb{F}}_p$-points of $\Spf \widehat{S}_n$. Write $x_{n_0}$ simply by $x_0$.  Then, it is enough to show that $\varphi$ is surjective after taking the stalk at $x$.   Let $S_{n,x_n}$ be the localization of $S_n$ at the prime $x_n$ and $S_{\infty, x}=\varinjlim_n S_{n, x_n}$.    Let $\widehat{S_{n,x_n}}$ be the completion of $S_{n,x_n}$ along its maximal ideal.   Recall that the ring $S_{n}$ is finite flat over $S$, this implies that  $\widehat{S_{n,x_n}}= S_{n,x_n}\otimes_{S_{x_0}} \widehat{S_{x_0}}$.

The scheme $\mc{n}$ is of finite type over $\bb{Z}_p$,  in particular every point has a closed point as specialization.  Thus,  by faithfully flat descent, we are reduced to prove that for every $\overline{\bb{F}}_p$-point  $x \in \Spf S_\infty/p= \varprojlim \Spf  S_n/p$,  the $\widehat{S_{x_0}}$-base change of 
\[
\varphi:  S_{\infty,x}/\pi\rightarrow S_{\infty,x}/p
\]
 is surjective (even an isomorphism).   We have the following commutative diagram 

\begin{equation*}
\begin{tikzcd}[row sep=small]
S_{\infty,x}/\pi \otimes_{S_{x_{0}}}\widehat{S_{x_{0}}} \arrow[r,"\varphi\otimes \id"] & S_{\infty,x}/p \otimes_{\varphi, S_{x_{0}}} \widehat{S_{x_{0}}}  \\
\varinjlim \widehat{S_{n,x_n}}/\pi \arrow[r] \arrow[u,equal]& \varinjlim(\widehat{S_{n,x_n}}/p \otimes_{\varphi, \widehat{S_{x_{0}}}}  \widehat{S_{x_{0}}}) \arrow[u,equal].
\end{tikzcd}  
\end{equation*} 
The ring $R_{n}$ is of finite type over $\bb{Z}_{p}$,  so that the absolute Frobenius $\varphi: R_{n}/\pi \rightarrow R_{n}/p $ is finite.   This implies that  $\widehat{S_{n,x_n}}/p$ is a finite $\widehat{S_{x_0}}$-module  via the module structure induced by the Frobenius.   Then,  the following composition is an isomorphism
\[
\widehat{S_{n,x_n}}/p \otimes_{\varphi, \widehat{S_{x_{0}}}} \widehat{S_{x_0}}\rightarrow  \varprojlim_m \widehat{S_{n,x_n}}/(p ,  \f{m}^{mp}_{S} ) \cong\widehat{S_{n,x_n}}/p ,
\]
where $\f{m}_S$  is  the maximal ideal of $S_{x_{0}}$.  
Thus, we are reduced   to prove  that the absolute Frobenius $\varphi: \varinjlim_n \widehat{S_{n,x_n}}/\pi \rightarrow \varinjlim_n \widehat{S_{n,x_n}}/p $ is surjective. Finally,  we deal with  the cusps, the supersingular and the ordinary points separately;  we use the  descriptions of   Section \ref{RecallsModularCurves}:

\begin{itemize}

\item \textbf{In the ordinary case},  the local ring $\widehat{S_{n,x_n}}$ is isomorphic to $ \breve{\bb{Z}}_p[\zeta_{p^n}][[X_n]]$. From the proof of  Proposition \ref{theodefringord}, one checks  that  the inclusion $\widehat{S_{n,x_n}}\rightarrow \widehat{S_{n+1,x_{n+1}}}$ is given by $X_{n}=(1+X_{n+1})^p-1$. Then,  one obtains the surjectivity of Frobenius when reducing modulo $p$. 

\item \textbf{The supersingular case}  is Corollary \ref{corsupersingfrob}. 

\item Finally, if we are dealing with a \textbf{cusp} $x$, the ring $\widehat{S_{n,x_n}}$ is isomorphic to $\breve{\bb{Z}}_p[\zeta_{p^{n}}][[q^{1/Np^{n}}]]$  and $\widehat{S_{n,x_n}}\rightarrow \widehat{S_{n+1,x_{n+1}}}$ is the natural inclusion by Proposition \ref{Theorem structure cusps 2}. The surjectivity of $\varphi$ is clear.
\end{itemize} 
\endproof

\subsection*{Relation with Lurie's stack}

In this subsection we make more explicitly  the relation between Lurie's construction of $\f{X}(Np^\infty)$ and the one presented in this document.  The key result   is the following  theorem 

\begin{theo}[ {\cite[Theo. 1.9]{JacobFullLevel2019}}]
Let $\pi\in \bb{Z}_p[\mu_{p^2}]$ be a pseudo-uniformizer such that $\pi^p=ap$  where $a$ is a unit.  For $n\geq 3$ there exists a unique morphism $\theta: \f{X}(Np^n)/ \pi \rightarrow \f{X}(Np^{n-1})/p$ making the following diagram commutative\footnote{The assumption $n\geq 3$ is only to guarantee that $\s{O}(\f{X}(Np^{n-1}))$ contains $\pi$.  }
\[
\begin{tikzcd}
\f{X}(Np^n)/p \ar[d] \ar[r, "\varphi"] & \f{X}(Np^n)/\pi  \ar[d] \ar[dl, "\theta"'] \\
\f{X}(Np^{n-1})/p \ar[r, "\varphi"] & \f{X}(Np^{n-1})/ \pi
\end{tikzcd}
\]
where $\varphi$ is the absolute Frobenius. 
\end{theo}

This theorem can be deduced from  the local computations made in Section \ref{RecallsModularCurves}.  Indeed, let $x_{n}\in \f{X}(Np^n)(\overline{\bb{F}}_p)$  be a $\overline{\bb{F}}_p$-point and $x_{n-1}\in \f{X}(Np^{n-1})(\overline{\bb{F}}_p)$ its image.  We have proven that there exists a unique map of the deformation rings at the points $x_{n-1}$ and $x_{n}$
\[
\theta^*: \widehat{\s{O}}_{\f{X}(Np^{n-1}), x_{n-1}} / p \rightarrow   \widehat{\s{O}}_{\f{X}(Np^{n}), x_{n}} / \pi
\]
making the following diagram commutative 
\[
\begin{tikzcd}
\widehat{\s{O}}_{\f{X}(Np^{n}), x_{n}} / p & \widehat{\s{O}}_{\f{X}(Np^{n}), x_{n}} / \pi  \ar[l,  "\varphi^*"']  \\
\widehat{\s{O}}_{\f{X}(Np^{n-1}),  x_{n-1}} / p   \ar[u] \ar[ru, "\theta^*"]& \widehat{\s{O}}_{\f{X}(Np^{n-1}),  x_{n-1}} / \pi   \ar[l ,  "\varphi^*"']  \ar[u]
\end{tikzcd}
\]
This corresponds  to Propositions \ref{theodefringord}, \ref{Theorem structure cusps 2} and \ref{propSuperSingFrob} for $x_{n}$ ordinary,  a cusp and a supersingular point respectively.  Then, one constructs $\theta$ using faithfully flat descent from the completed local rings to the localized local rings at $x_{n}$, and glueing using the uniqueness  of  $\theta^*$. 

\section{Cohomology and local duality for  curves over $\n{O}_K$}

\label{Reminders Duality}

Let $K$ be a finite extension of $\bb{Q}_p$ and $\n{O}_K$ its valuation ring. In this section we recall the Grothendieck-Serre duality theorem for  local complete intersection (lci) projective curves over $\n{O}_K$,  we will follow \cite{hartshorne1966residues}.   Then, we use Pontryagin duality to define a local duality paring of coherent cohomologies.

Let $X$ be a locally  noetherian scheme and  $\D(X)$  the derived category of $\s{O}_X$-modules. We use subscripts $c, qc$ on $\D(X)$ for the derived category of $\s{O}_{X}$-modules with coherent and quasi-coherent cohomology, the subscript $fTd$  refers to the subcategory of complexes with finite Tor dimension. We use superscripts $+,-,b$ for the derived category of bounded below, bounded above and bounded complexes respectively. For instance, $\D_{c}^b(X)_{fTd}$ is the derived category of bounded complexes of $\s{O}_{X}$-modules of finite Tor dimension and coherent cohomology. If $X=\Spec A$ is affine, we set $\D(A):=\D_{qc}(X)$,  the derived category of $A$-modules.

\begin{definition}
Let  $f:X\rightarrow Y$ be a morphism of schemes.  
\begin{enumerate}
\item  The map  $f$  is \textit{embeddable} if it factors as $X\xrightarrow{\iota} S \rightarrow Y $ where $\iota$ is a finite morphism and $S$ is smooth over $Y$. 
\item The map $f$  is \textit{projectively embeddable} if  it factors as composition $ X\xrightarrow{\iota} \bb{P}^n_Y\rightarrow Y$ for some $n\geq 0$,  where $\iota$ a finite morphism. 

\item The map $f$ is a \textit{local complete intersection} if locally on $Y$ and $X$ it factors as $X\xrightarrow{\iota} S \rightarrow Y $, where $S$ is a smooth $Y$-scheme,  and $\iota$ is a closed immersion defined by a regular sequence of $S$.  The  length of the regular sequence is called  the codimension of $X$ in $S$.   
\end{enumerate}
\end{definition}

\begin{theo}[Hartshorne]
\label{GrothSerreDuality}

 Let $f: X\rightarrow Y$ be a projectively embeddable morphism of  noetherian schemes of  finite Krull dimension.  Then there exist an exceptional inverse image functor $f^!: \mathrm{D}(Y)\rightarrow \mathrm{D}(X)$,  a trace map  $\Tr: \mathrm{R}f_*f^!\rightarrow 1$ in  $\mathrm{D}^+_{qc}(Y)$,  and  an adjunction 
\[
\theta: \R f_* \R \s{H}om_{X}(\n{F}, f^! \n{G}) \xrightarrow{\sim}  \R\s{H}om_Y( \R f_* \n{F}, \n{G}) 
\]
for $\s{F}\in \mathrm{D}^-_{qc}(X)$ and $\s{G}\in \mathrm{D}^+_{qc}(Y)$.

Moreover, the formation of the exceptional inverse image is functorial. More precisely, given a composition $\begin{tikzcd}  X \ar[r,"f"]& Y \ar[r,"g"]& Z\end{tikzcd}$ with $f,g$ and $gf$ projectively embeddable,  there is a natural isomorphism $(gf)^!\cong f^!g^!$. This functor  commutes with flat base change.  Namely,  let $u:Y'\rightarrow Y$ be a flat morphism, $f':X'\rightarrow Y'$ the base change of $X$ to $Y'$ and $v: X'\rightarrow X$ the projection. Then there is a natural isomorphism of functors $v^*f^!={f'}^! u^*$. 
\proof
 We refer to   \cite[Theo. III. 8.7]{hartshorne1966residues} for the existence of $f^!$,  its functoriality and compatibility with flat base change.  See Theorems    III. 10.5 and  III 11.1 of \textit{loc. cit.}  for the existence of $\Tr$ and the adjunction $\theta$ respectively.     
\endproof
\end{theo}

\begin{exam} 
\label{ExampleDualizingSheaf}
Let $f:X\rightarrow Y$ be a morphism of  finite type of   noetherian schemes of finite Krull dimension. 
\begin{enumerate}

\item  We can define the functor $f^!$ for finite  morphisms as 
\[
f^!\n{F}=f^{-1} \R \s{H}om_{\n{O}_Y}(f_*\s{O}_X, \n{F}) \mbox{ for } \n{F}\in \D(Y).
\]
The duality theorem  in this case is  equivalent to the (derived) $\otimes$-$\Hom$ adjunction, see \cite[\S III. 6]{hartshorne1966residues}. 
\item   Let $f$ be smooth of relative dimension $n$, then one has $f^!\n{F}=\n{F}\otimes \omega_{X/Y}^{\circ}[n]$ where $\omega_{X/Y}^{\circ}=\bigwedge^n \Omega^1_{X/Y}$,  see\cite[\S III.2]{hartshorne1966residues}.  
\end{enumerate}

\end{exam}

\begin{lem}
\label{fOy}
Let $f: X\rightarrow Y$ be an lci morphism of relative dimension $n$ between locally noetherian schemes of finite Krull dimension. Then $f^!\s{O}_Y = \omega^{\circ}_{X/Y}[n]$ with $\omega^{\circ}_{X/Y}$ an invertible $\s{O}_X$-module. 
\proof
Working locally on $Y$ and $X$,  we may assume that $f$ factors as $\begin{tikzcd}[column sep= 1.30 em]  X \ar[r, "\iota"] & S \ar[r, "g"] & Y \end{tikzcd}$, where $g$ is a smooth morphism of  relative dimension $m$, and $\iota$ is a regular closed imersion of codimension $m-n$ defined by an ideal $\s{I}=(f_1,\ldots, f_{m-n})$. Let $\omega^{\circ}_{S/Y}=\bigwedge^m\Omega^{1}_{S/Y}$ be the sheaf of $m$-differentials of $S$ over $Y$, then 
\begin{eqnarray*}
f^!\s{O}_Y & = & \iota^! g^! \s{O}_Y \\ 
		   & = & \iota^{-1} \R \s{H}om_{\s{O}_S}( \iota_* \s{O}_X,  \omega_{S/Y}^\circ[m]) \\
		   & = &\iota^{-1} \R \s{H}om_{\s{O}_S}(\s{O}_S/\s{I},  \omega_{S/Y}^\circ)[m]
\end{eqnarray*}
Let $K(\underline{f})$ be the Koszul complex of the regular sequence $\underline{f}=(f_1,\ldots, f_{m-n})$. Then $K(\underline{f})$ is a flat resolution of $\s{O}_S/\s{I}$, its dual $K(\underline{f})^\vee=\s{H}om_{\s{O}_S}(K(\underline{f}), \s{O}_S)$ is a flat resolution of \mbox{\((\s{O}_S/\s{I})[-(m-n)]\)}. Therefore 
\begin{eqnarray*}
f^!\s{O}_Y & \simeq & \iota^{-1}  \s{H}om_{\s{O}_S}(K(\underline{f}), \omega^\circ_{S/Y})[m] \\ 
		   & = & \iota^{-1}  K(\underline{f})^\vee \otimes \omega^\circ_{S/Y} [m]  \\ 
		   & \simeq & \iota^{-1}(  \s{O}_{S}/\s{I}\otimes ^L \omega^\circ_{S/Y}[n]) \\
		   & = & \iota^{-1}((\omega^\circ_{S/Y}/\s{I})[n]) = \iota^* \omega^{\circ}_{S/Y}[n] 
\end{eqnarray*}
which is an invertible sheaf of $\s{O}_X$-modules as required. 
\endproof
\end{lem}

\begin{remark}
Let $f:X\rightarrow Y$ be a regular closed immersion of codimension $n$ defined by the ideal $\s{I}$.  From the proof of Lemma \ref{fOy} one can deduce that $f^!\s{O}_Y=\bigwedge^n f^*(\s{I}/\s{I}^2)^{\vee}[-n]$ is the normal sheaf concentrated in degree $n$.

\end{remark}

The compatibility of $f^!$ with tensor products allows us to compute $f^!\s{F}$ in terms of $f^*\n{F}$ and $f^!\s{O}_Y$:

\begin{prop}[\!{\cite[Prop. III.8.8]{hartshorne1966residues}}]
\label{dualizingAndTensor}
Let $f:X\rightarrow Y$ be an embeddable morphism of locally noetherian schemes of finite Krull dimension. Then there are functorial isomorphisms 
\begin{enumerate}

\item $f^!\n{F}\dotimes \mathrm{L}f^*\n{G}\rightarrow  f^!(\n{F}\otimes^{\mathrm{L}} \n{G})$ for $\n{F}\in \mathrm{D}_{qc}^+(Y)$ and $\n{G}\in \mathrm{D}^b_{qc}(Y)_{fTd}$.

\item  $\mathrm{R} \s{H}om_X(\mathrm{L}f^*\n{F}, f^! \n{G})\rightarrow f^!(\R \s{H}om_{Y}(\n{F}, \n{G}))$ for $\n{F}\in \D^-_{c}(Y)$ and $\n{G}\in \D^+_{qc}(Y).$

\end{enumerate}

Moreover, if $f$ is an lci morphism,  then   $f^!\s{O}_Y$ is invertible and  we have $f^!\n{G}\cong f^!\s{O}_Y\otimes \mathrm{L}f^* \n{G}$ for $\n{G}\in \D^{b}_{qc}(Y)_{fTd}$.  We call $f^! \s{O}_Y$ the dualizing sheaf of $f$. 
\end{prop}

We now prove the local duality theorem for  vector bundles over lci projective  curves:

\begin{prop}
\label{LocalDualityCurves}
 Let $f: X\rightarrow \Spec \n{O}_K$ be an lci  projective  curve, and let $\omega_{X/\n{O}_K}^\circ$ be the  dualizing sheaf of $f$,  i.e. the invertible sheaf such that $\omega_{X/\n{O}_K}^\circ[1]=f^!\n{O}_K$. Let $\n{F}$ be a locally free $\s{O}_X$-module of finite rank, then:  
\begin{enumerate}

\item $\R f_*\n{F}$ is representable by a perfect complex of lenght $[0,1]$;

\item we have  a perfect pairing 
\begin{align*}
\mathrm{H}^0(X,\n{F}\otimes K/\n{O}_K)\times \mathrm{H}^1(X, \n{F}^\vee\otimes \omega_{X/\n{O}_K}^\circ)\rightarrow K/\n{O}_K
\end{align*}
\end{enumerate}
given by the composition of the cup product and the trace $\Tr: Rf_* \omega^{\circ}_{X/\n{O}_K}\rightarrow \n{O}_K$. 
\proof
As $\n{F}$ is a vector bundle and $f$ is projective of relative dimension $1$, the cohomology groups  $\R^i f_* \n{F}$ are finitely generated over $\n{O}_K$ and concentrated in degrees $0$ and $1$. Then,  $\R f_* \n{F}$ is quasi-isomorphic to a complex $\begin{tikzcd}[column sep =  1.30 em] 0 \ar[r] & M_0 \ar[r, "d"]& M_1 \ar[r]& 0 \end{tikzcd}$ with $M_1$ and $M_2$ finite free $\n{O}_K$-modules.  Moreover, the complex $\begin{tikzcd}[column sep = 1.30em] 0 \ar[r] & M_0\otimes K/\n{O}_K \ar[r, "d\otimes 1"]& M_1\otimes K/\n{O}_K \ar[r]& 0 \end{tikzcd}$ is quasi-isomorphic to $\R f_*(\n{F}\otimes K/\n{O}_K)$ in $\D(\n{O}_K)$, see \cite[Theo. 5.2]{mumford1974abelian}. 

Duality theorem \ref{GrothSerreDuality} gives  a quasi-isomorphism 
\[
\R f_*(\n{F}^{\vee}\otimes \omega^{\circ}_{X/\n{O}_K})[1] =\R f_*\R\s{H}om_{X}(\n{F}, f^!\n{O}_K)\simeq \R\Hom_{\n{O}_K}(\R f_*\n{F}, \n{O}_K).
\]
 This implies that $\R f_* (\n{F}^\vee \otimes \omega_{X/\n{O}_K}^\circ)$ is quasi-isomorphic to $\begin{tikzcd}[column sep =  1.30 em] 0 \ar[r] & M_1^\vee \ar[r, "d^\vee"]& M_0^\vee \ar[r]& 0 \end{tikzcd}$. Finally, Pontryagin duality for $\n{O}_K$ implies $\Hom_{\n{O}_K}(\ker(d\otimes 1), K/\n{O}_K  )= \coker d^\vee$,  which translates in the desired statement. 
\endproof

\end{prop}

\begin{remark}
\label{Remarkduality}
The previous proposition relates  two  notions of duality. Namely,  Serre and Pontryagin duality.  We can deduce the following facts: 

\begin{enumerate}

\item The $\n{O}_K$-module $\mathrm{H}^0(X, \n{F}\otimes K/\n{O}_K)$ is co-free of rank $r$, that is isomorphic to $(K/\n{O}_K)^r$, if and only if  $\mathrm{H}^1(X, \n{F}^\vee \otimes \omega_{X/\n{O}_K}^{\circ})$ is free of rank $r$. In that case, the module $\mathrm{H}^0(X, \n{F})$ is free and  $\mathrm{H}^0(X, \n{F})/p^n \rightarrow \mathrm{H}^0(X,\n{F}/p^n)$ is an isomorphism for all $n\in \bb{N}$. Furthermore,  Serre duality provides a perfect pairing 
\[
\mathrm{H}^0(X, \n{F}) \times \mathrm{H}^1(X, \n{F}^\vee \otimes \omega^\circ_{X/\n{O}_K})\rightarrow \n{O}_K.
\]

\item  The $\n{O}_K$-module $\mathrm{H}^0(X,\n{F})$ (resp. $\mathrm{H}^1(X, \n{F}\otimes K/\n{O}_K)$) is free (resp. co-free) for any finite locally free $\s{O}_X$-module. 

\item  In the notation of  the previous proof, Pontryagin duality implies \[\Hom_{\n{O}_K}(\coker(d\otimes 1), K/\n{O}_K)=\ker d^{\vee},\] which is equivalent to a perfect pairing  
\[
\mathrm{H}^1(X, \n{F}\otimes K/\n{O}_K) \times \mathrm{H}^0(X,\n{F}^\vee\otimes \omega_{X/\n{O}_K}^\circ)\rightarrow K/\n{O}_K.
\]
\end{enumerate}

\end{remark}

\section{Cohomology of modular sheaves}
\label{COhomology of modular sheaves}

Let $N\geq 3$ be an integer prime to $p$.  Let $X(Np^n)$ be the modular curve over $\bb{Z}_p$ of level $\Gamma(Np^n)$.  Let $\breve{\bb{Z}}_p=W(\overline{\bb{F}}_p)$ and  let $X(Np^n)_{\breve{\bb{Z}}_p}$ be  the extension of scalars of $X(Np^n)$ to $\breve{\bb{Z}}_p$. We denote by $X_n:=X(Np^n)^{\circ}_{\breve{\bb{Z}}_p}$ the connected component  of $X(Np^n)_{\breve{\bb{Z}}_p}$ given by fixing the Weil pairing $e_{N}(P_N,Q_N)=\zeta_{N}$, where $(P_N,Q_N)$ is the universal basis of $E[N]$ and $\zeta_N\in \breve{\bb{Z}}_p$ a primitive $N$-th root of unity. We also write $X=X_0$.   Let $\n{O}_n= \breve{\bb{Z}}_p[\mu_{p^n}]$ be the $n$-th   cyclotomic extension of $\breve{\bb{Z}}_p$,  $\n{O}^{cyc}$ the $p$-adic completion of $\varinjlim_{n} \n{O}_n$,  $K_n$ and $K^{cyc}$ the field of fractions of $\n{O}_n$  and $\n{O}^{cyc}$ respectively.  We set $\n{O}=\breve{\bb{Z}}_p$ and  $K=\n{O}[\frac{1}{p}]$.  Let $\pi_n : X_n\rightarrow \Spec \n{O}_n$ denote the structural map defined by the Weyl pairing of the universal basis of $E[p^n]$.  We also denote $p_n: X_n\rightarrow X_{n-1}$ the natural morphism induced by $p$-multiplication of Drinfeld bases. 

Let  $E^{sm}/X$ be the semi-abelian scheme over $X$ extending the universal elliptic curve to the cusps,  cf  \cite{deligne1973schemas}. Let $e:X\rightarrow E^{sm}$ be the unit section and $\omega_E:=e^* \Omega^1_{E^{sm}/X}$ the modular sheaf, i.e., the sheaf of invariant differentials of $E^{sm}$ over $X$.  For $k\in \bb{Z}$ we define  $\omega_E^{k}= \omega_E^{\otimes k}$ the sheaf of modular forms of weight $k$, we denote by $\omega^k_{E,n}$ the pullback of $\omega_E^k$ to $X_n$.  Let $D_n\subset X_n$ be the (reduced) cusp divisor  and $\omega_{E,n,cusp}^k:=\omega_{E,n}^k(-D_n)$ the sheaf of cusp forms of weight $k$ over $X_n$.      By an abuse of notation we will also write $D_n$ for the pullback $p_{n+1}^* D_n$ to $X_{n+1}$,  by Proposition \ref{Theorem structure cusps 2} we have that $D_{n}=p D_{n+1}$. 

Finally, we let $\f{X}_n$ be the completion of $X_n$ along its special fiber and $\f{X}_{\infty}=\varprojlim_{n} \f{X}_n$ the integral perfectoid modular curve,  see Theorem \ref{theoexistperfint}.  Let $\n{X}_n$ be the analytic generic fiber of $X_n$ and $\n{X}_{\infty}\sim  \varinjlim_n \n{X}_n$ the Scholze's perfectoid modular curve.

\subsection*{Dualizing sheaves of modular curves}

 Consider the tower of modular curves 
\[
\begin{tikzcd}
\cdots \ar[r] & X_{n+1} \ar[r,"p_{n+1}"] \ar[d, "\pi_{n+1}"]& X_n \ar[r,"p_n"] \ar[d, "\pi_n"]& X_{n-1} \ar[r] \ar[d, "\pi_{n-1}"] & \cdots \\
\cdots \ar[r] & \Spec(\n{O}_{n+1}) \ar[r] & \Spec(\n{O}_{n}) \ar[r] & \Spec(\n{O}_{n-1}) \ar[r] & \cdots.
\end{tikzcd}
\]

 Since $X_n$ is regular of finite type over $\n{O}_n$, it is a local complete intersection.  This implies that the sheaf  $\omega^{\circ}_n:= \pi_{n}^! \n{O}_n$ is invertible.  
The modular curve $X/\n{O}$ is smooth of relative dimension 1,  then we have that  $\omega_{0}^\circ = \Omega_{X_0/\n{O}}^1$, cf. Example \ref{ExampleDualizingSheaf} (2).   On the other hand,   the Kodaira-Spencer map  provides an isomorphism $KS:\omega^2_{E,cusp}\cong \Omega^1_{X/\n{O}}$.

Let $X'_{n-1}= X_{n-1}\times_{\Spec \n{O}_{n-1}}\Spec \n{O}_n$, and by an abuse of notation  $p_n: X_{n}\rightarrow X'_{n-1}$ the induced map.  Let $\pi'_{n-1}: X'_{n-1}\rightarrow \n{O}_n$ be the structural map and $pr_1: X'_{n-1}\rightarrow X_{n-1}$ the first projection.  We also write $\omega_{E,n-1}^k$ for the pullback of $\omega_{E,n-1}^k$ to $X'_{n-1}$.  Note that the compatibility of the exceptional inverse image functor with flat base change (Theorem \ref{GrothSerreDuality}) implies that $\pi^{'!}_{n-1} \n{O}_n\cong pr_1^{*} \omega_{n-1}^{\circ}= \omega_{n-1}^{\circ}\otimes_{\n{O}_{n-1}}\n{O}_n$.

\begin{prop}

\label{isoDualizingModforms}

There exists a natural isomorphism $\xi_n: p_{n}^*(\omega_{n-1}^{\circ})( D_{n-1}-D_n)\xrightarrow{\sim} \omega_n^{\circ}$ induced by the normalized trace $\frac{1}{p}\Tr_n:\s{O}_{X_n}\rightarrow p_n^! \s{O}_{X'_{n-1}}$.  Moreover, the composition of    $ \xi_n \circ \cdots \circ \xi_1$  with the Kodaira-Spencer map gives an isomorphism $\omega^2_{E,n,cusp}\cong \omega^{\circ}_n$.

\proof

 By Proposition \ref{dualizingAndTensor} we have an isomorphism
\begin{equation}
\label{equalityDualizingSheaves}
\xi'_n: p^!_n\s{O}_{X'_{n-1}}\otimes p_n^* \omega^\circ_{n-1}\xrightarrow{\sim} p_n^! \omega_{n-1}^\circ=\omega_{n}^\circ.
\end{equation}
The map $p_n$ is finite flat, then  $p_n^!\s{O}_{X'_{n-1}}= p_n^{-1}\s{H}om_{\s{O}_{X'_{n-1}}}(p_{n,*} \s{O}_{X_n}, \s{O}_{X'_{n-1}})$ by  Example \ref{ExampleDualizingSheaf} (1).  By Lemma \ref{fOy},  the sheaf $p_n^! \s{O}_{X'_{n-1}}$ is invertible as $X'_{n-1}$ is an lci projective curve.  We claim that the trace $\Tr_n:\s{O}_{X_n}\rightarrow p_n^! \s{O}_{X'_{n-1}}$ induces an isomorphism $\frac{1}{p} \Tr_n:  \s{O}_{X_n}(D_{n-1}-D_n)\cong p^!_n \s{O}_{X'_{n-1}}$.   It suffices to consider  the ordinary points and the cusps,  indeed,  the supersingular points are of codimension $2$ in $X_n$.  

Let  $x\in X'_{n-1}(\overline{\bb{F}}_p)$ be an ordinary point.  We have  a cartesian square
\begin{equation}
\label{squareLocalizations}
\begin{tikzcd}
\bigsqcup_{x_n\mapsto x} \Spf \widehat{\s{O}}_{X_n, x_n} \ar[r] \ar[d]& X_n\ar[d,"p_n"] \\
\Spf \widehat{\s{O}}_{X'_{n-1},x} \ar[r]& X'_{n-1}.
\end{tikzcd}
\end{equation}
By Proposition \ref{theodefringord} we have isomorphisms 
\[
\widehat{\s{O}}_{X'_{n-1},x}\cong W(\overline{\bb{F}}_p)[\zeta_{p^{n}}][[T_{n-1}]], \;\;\; \widehat{\s{O}}_{X_n,x_n}\cong W(\overline{\bb{F}}_p)[\zeta_{p^{n}}][[T_{n}]]
\]
with relations $(1+T_n)^p=1+T_{n-1}$.  Taking the different ideal of the finite flat extension $\widehat{\s{O}}_{X_n,x_n} / \widehat{\s{O}}_{X'_{n-1},x}$,  one finds 
\[
\s{H}om_{\widehat{\s{O}}_{X'_{n-1},x}}(\widehat{\s{O}}_{X_n,x_n},\widehat{\s{O}}_{X'_{n-1},x})\cong \frac{1}{p}\widehat{\s{O}}_{X_n,x_n}\cdot \Tr_{n} . 
\]

On the other hand, let $x\in X'_{n-1}(\overline{\bb{F}}_p)$ be a  cusp.  We have a cartesian square (\ref{squareLocalizations})  and by Proposition \ref{Theorem structure cusps 2} isomorphisms 
\[
\widehat{\s{O}}_{X'_{n-1},x}\cong W(\overline{\bb{F}}_p)[\zeta_{p^{n}}][[q^{1/p^{n-1}}]], \;\;\; \widehat{\s{O}}_{X_n,x_n}\cong W(\overline{\bb{F}}_p)[\zeta_{p^{n}}][[q^{1/p^{n}}]].
\]
Taking  the different ideal we obtain the equality 
\begin{eqnarray*}
\s{H}om_{\widehat{\s{O}}_{X'_{n-1},x}}(\widehat{\s{O}}_{X_n,x_n},\widehat{\s{O}}_{X'_{n-1},x})& \cong & \frac{1}{p} q^{1/p^n-1/p^{n-1}} \widehat{\s{O}}_{X_n,x_n}\cdot \Tr_n.
\end{eqnarray*}
The previous computations show that  the trace of $\s{O}_{X_n}/\s{O}_{X'_{n-1}}$ induces an isomorphism of invertible sheaves
\[
 \frac{1}{p} \Tr_n: \s{O}_{X_n}(D_{n-1}- D_n) \xrightarrow{\sim}  p_n^!\s{O}_{X'_{n-1}}. 
\]
Then, from (\ref{equalityDualizingSheaves}) we have an isomorphism
\[
\xi_{n}:  \s{O}_{X_n}( D_{n-1}- D_n)\otimes p_n^*\omega^\circ_{n-1}\rightarrow  \omega^\circ_n
\]
with $\xi_n= \xi'_n\circ ( \frac{1}{p} \Tr_n \otimes 1)$.

The isomorphism $\omega_{E,n,cusps}^2\cong \omega_n^\circ$ follows by a straightforward induction on the composition $\xi_n\circ \cdots \circ \xi_{1}$,  and the Kodaira-Spencer map $KS: \omega_{E,cusp}^2\cong \Omega^1_{X/\n{O}}$. 
\endproof

\end{prop}

\begin{lem}
\label{LemmaSurjTrace}
Let $x\in X'_{n-1}(\overline{\bb{F}}_p)$ be an  an ordinary point or a cusp.  Let $\widetilde{\Tr}_n:  p_{n,*}\s{O}_{X_n}(D_{n-1}-D_n)\rightarrow \s{O}_{X'_{n-1}}$ be the  normalized  trace map $\frac{1}{p} \Tr_n$.  Then the completed  localization of $\widetilde{\Tr}_n$ at $x$  is surjective.  Moreover,   if  $\n{F}$ is a quasi-coherent sheaf  over $X'_{n-1}$,  the composition $\n{F}\rightarrow p_{n,*}p_n^*\n{F}\xrightarrow{\widetilde{\Tr}_n} \n{F}$ is multiplication by $p$.
\proof
Localizing at $x$ we find 
\[
\widetilde{\Tr}_n= \oplus (\frac{1}{p}\Tr_n):\bigoplus_{x_n\mapsto x} \widehat{\s{O}}_{X_n, x_n}\otimes (q^{1/p^n-1/p^{n-1}}) \rightarrow  \widehat{\s{O}}_{X'_{n-1},  x}
\]
where $q^{1/p^{n-1}}$ is invertible if $x$ is ordinary, or a generator of $D_{n-1}$ if it is a cusp.   The explicit descriptions found in the previous proposition show that $\widetilde{\Tr}_n$ is surjective on each direct summand.   Finally,  looking at an ordinary point $x$,  it is clear that there are $p$ different points $x_n$ in the fiber of $x$, this implies $\widetilde{\Tr}_n(1)=p$. 
\endproof
\end{lem}

\subsection*{Vanishing of coherent cohomology}

In order to prove vanishing theorems for the coherent cohomology over the  perfectoid modular curve,  we first need some vanishing results at finite integral level.  We have the following proposition

\begin{prop}
\label{vanishingNeg}
For all  $n\in \bb{N}$ the following holds

\begin{enumerate} 

\item $\mathrm{H}^0(X_n,\omega_{E,n}^k\otimes_{\n{O}_n} K_n/\n{O}_n)=\mathrm{H}^0(X_n, \omega^k_{E,n,cusp}\otimes_{\n{O}_n} K_n/\n{O}_n)=0$ for $k<0.$

\item  $\mathrm{H}^1(X_n,\omega_{E,n}^k)=\mathrm{H}^1(X_n, \omega^k_{E,n,cusp})=0$ for $k>2$.

\item $\mathrm{H}^0(X_n, \s{O}_{X_n}(-D_n)\otimes_{\n{O}_n} K_n/\n{O}_n)=\mathrm{H}^1(X_n,\omega^2_{E,n})=0$ and $\mathrm{H}^0(X_n, \s{O}_{X_n}\otimes_{\n{O}_n} K_n/\n{O}_n)=  \mathrm{H}^1(X_n, \omega^2_{E,n,cusp})\otimes_{\n{O}_n}  (K_n/\n{O}_n)= K_n/\n{O}_n$.

\end{enumerate}

\proof

By Propositions \ref{LocalDualityCurves} and \ref{isoDualizingModforms},    (1) and (2) are equivalent.  Similarly, by   (1)    of Remark \ref{Remarkduality},  and Proposition \ref{isoDualizingModforms}, it is enough to show (3) for $\s{O}_{X_n}$ and $\s{O}_{X_n}(-D_n)$. 

 Let   $\nu_n$ be the closed point of $\Spec \n{O}_n$  and $\varpi\in \n{O}_n$ a uniformizer, we write $\nu=\nu_0$ for the closed point of $\Spec \n{O}$.   It suffices to prove $\mathrm{H}^0(X_{n}, \omega^k_{E,n}/\varpi)=\mathrm{H}^0(X_{n,\nu_n}, \omega^k_{E,n})=0$ for $k<0$.  Indeed,  for $s\geq 1$,  the short exact sequence \[0\rightarrow \omega^k_{E,n}/ \varpi^{s}\xrightarrow{\varpi} \omega^k_{E,n}/\varpi^{s+1}\rightarrow \omega^{k}_{E,n}/\varpi \rightarrow 0\] induces a left exact sequence in global sections
 \[
 0\rightarrow \mathrm{H}^0(X_n, \omega^k_{E,n}/ \varpi^s)\rightarrow \mathrm{H}^0(X_n, \omega^k_{E,n}/\varpi^{s+1})\rightarrow \mathrm{H}^0(X_n, \omega^k_{E,n}/\varpi).
 \] 
  An inductive argument on $s$ shows $\mathrm{H}^0(X_n, \omega^k_{E,n}/\varpi^s)=0$ for all $s\geq 1$.

Let $\lambda\in \mathrm{H}^0(X_{n,\nu_n}, \omega^k_{E,n})$ be non-zero.  Applying the action of  $\SL_2(\bb{Z}/p^n\bb{Z})$, we can assume that $\lambda$ is non-zero in an open dense subscheme of $X_{n,\nu_n}$. In fact, this holds for some linear combination $\sum_{\gamma\in \SL_2(\bb{Z}/p^n\bb{Z}_p)} a_n \gamma^* \lambda$ with $a_n\in \overline{\bb{F}}_p$. The norm $\N_{X_{n,\nu_n}/X_{\nu}}(\omega^k_{E,n})$ of $\omega^k_{E,n}$ to $X_\nu$ is $\omega^{kd}_{E}$,  where $d=\deg(X_{n,\nu_n}/X_\nu)$.   Hence if $k<0$,  the sheaf $\N_{X_{n,\nu_n}/X_\nu}(\omega^k_{E,n})$ has negative degree in the smooth curve $X_{\nu}$.  This implies that $\mathrm{H}^0(X_{\nu}, \N_{X_{n,\nu_n}/X_\nu}(\omega^k_{E,n}))=0$ and $\N_{X_{n,\nu_n}/X_{\nu}}(\lambda)=0$,  a contradiction. Therefore  $\mathrm{H}^0(X_{n,\nu_n}, \omega^k_{E,n})=0$ for $k<0$.  Since $\omega^k_{E,n,cusp}=\omega_{E,n}^k(-D_n)$, we trivially deduce $\mathrm{H}^0(X_{n,\nu_n},\omega^k_{E,n,cusp})=0$. 

The results for $\s{O}_{X_n}$ and $\s{O}_{X_n}(-D_n)$ are clear as $X_n/\n{O}_{n}$ is proper, flat,  geometrically connected  and has geometrically  reduced fibers.  

\endproof
\end{prop}

\begin{remark}
\label{remark1}
Strictly speaking, we can apply Proposition \ref{LocalDualityCurves} only for projective  curves over a finite extension of $\bb{Z}_p$.  However, as the formation of coherent cohomology is compatible with affine  flat base change of the base, the conclusion of \textit{loc. cit.} holds in the situation of the previous proposition. 
\end{remark}

\begin{cor}
\label{localDualcoro}
Let $\n{F}=\omega^k_{E,n}$ or $\omega^k_{E,n,cusp}$ for $k\neq 1$,  the following holds
\begin{enumerate}

\item The cohomology groups $\mathrm{H}^0(X_n, \n{F}\otimes K/\n{O})$ and $\mathrm{H}^1(X_n, \n{F})$ are cofree and free $\n{O}_n$-modules respectively. 

\item We have a perfect duality pairing 
\[
\mathrm{H}^0(X_n, \n{F}\otimes K/\n{O})\times \mathrm{H}^1(X_n, \n{F}^\vee \otimes \omega^2_{n,cusp})\rightarrow K_n/\n{O}_n. 
\]

\end{enumerate}

\proof
Part (2) is Proposition \ref{LocalDualityCurves}.  Part (1) follows from Remark \ref{Remarkduality} i)   and the previous proposition.   Indeed,  if $k<0$,  the vanishing of $\mathrm{H}^0(X_n, \n{F}\otimes K_n/\n{O}_n)$ implies that $\mathrm{H}^1(X_n, \n{F})$  is torsion free.   As the cohomology group  is of finite type over $\n{O}_n$, it is a finite free $\n{O}_n$-module.  The other cases are proved in a similar way.  

\endproof

\end{cor}

Next, we will prove some cohomological vanishing  results for the modular sheaves $\omega_E^k$ and $\omega^k_{E,cusp}$  at infinite level.  Particularly,   we will show that the cohomology of $\omega_E^k$ over $\f{X}_{\infty}$ is concentrated in degree $0$ if $k>0$.    The case  $k> 2$ will follow  from  Proposition \ref{vanishingNeg},  one can also argue directly for $k=2$.   What is remarkable is the vanishing for $k=1$, in which case we use the perfectoid nature of $\f{X}_{\infty}$.  

 Let $\omega_{E,\infty}^{k}$ be the pullback of $\omega_E^k$ to $\f{X}_{\infty}$.   Let $m\geq n$,  note that  we have an inequality of divisors $D_m\leq D_n$.  Then,  $\s{O}_{X_m}(-D_n)\subset \s{O}_{X_m}(-D_m)$, and  the pullback of $\omega^k_{E,n,cusp}$ injects into $\omega^k_{E,m,cusp}$.   We define $\omega^{k}_{E,\infty,cusp}$ as  the $p$-adic completion of the direct limit $\varinjlim_n \omega^k_{E,n,cusp}$, if $k=0$ we simply write $\s{O}_{\f{X}_{\infty}}(-D_{\infty})$ for $\omega^0_{E,\infty, cusp}$.    The sheaf $\omega^k_{E,\infty,cusp}$ is no longer a coherent sheaf over $\f{X}_\infty$; its reduction modulo $p$ is a direct limit of line bundles which is not stationary at the cusps.   One way to think about an element in $\omega^k_{E,\infty,cusp}$ is via $q$-expansions:   the completed localization of $\omega^k_{E,\infty}$ at a cusp $x=(x_0,x_1,\cdots)\in \f{X}_{\infty}$ is isomorphic to  
\[
\n{O}^{cyc}[[q^{1/p^\infty}]]:= \varprojlim_{s}  ( \varinjlim_n  \n{O}^{cyc}[[q^{1/p^n}]])/(p,q)^s.
\] 
Then,  an element $f\in \omega^k_{E,\infty,x}$ can be written as a power series 
\[
f= \sum_{m\in \bb{Z}[\frac{1}{p}]_{\geq 0}} a_m q^{m}
\]
satisfying certain convergence conditions.   The element $f$ belongs to the localization at $x$ of $\omega^k_{E,\infty, cusp}$ if and only if $a_0=0$.  For a  detailed treatment of the cusps at perfectoid level we refer to \cite{heuer2020cusps},  particularly Theorem 3.17.

\begin{theo}
\label{corVanishingCohPerfectoid}
The following holds

\begin{enumerate}

\item The cohomology complexes $\R\Gamma(\f{X}_{\infty}, \omega^k_{E,\infty})$ and $\R\Gamma(\f{X}_{\infty}, \omega^k_{E,\infty,cusp})$ are concentrated in degree $[0,1]$ for all $k\in \bb{Z}$.   

\item For all $m,i\geq 0$ and $k\in \bb{Z}$, we have  $\mathrm{H}^i(\f{X}_{\infty}, \omega^k_{E,\infty}/p^m)=\varinjlim_n \mathrm{H}^i(X_n, \omega^k_{E,n}/p^m)$ and  $\mathrm{H}^i(\f{X}_{\infty}, \omega^k_{E,\infty,cusp}/p^m)=\varinjlim_n \mathrm{H}^i(X_n, \omega^k_{E,n,cusp}/p^m)$.

\item  The sheaves $\omega_{E,\infty}^{k}$ and $\omega_{E,\infty,cusp}^k$ have cohomology concentrated in degree $0$ for $k>0$.  Similarly,   the sheaves $\omega_{E,\infty}^{k}$ and $\omega_{E,\infty,cusp}^{k}$ have cohomology concentrated in degree $1$ for $k<0$.

\item $\mathrm{H}^0(\f{X}_{\infty}, \s{O}_{\f{X}_{\infty}}(-D_{\infty}))=0$ and  $\mathrm{H}^0(\f{X}_{\infty}, \s{O}_{\f{X}_{\infty}})= \n{O}^{cyc} $.

\end{enumerate}

\proof
Let $\n{F}= \omega^k_{E,\infty}$ or $\omega^{k}_{E,\infty,cusp}$ and $\n{F}_n=\omega^k_{E,n}$ or $\omega^k_{E,n,cusp}$ respectively.  We show (1) assuming part  (2).  By evaluating $\n{F}$ at formal affine perfectoids of $\f{X}_{\infty}$ arising from finite level,  one can use Lemma 3.18 of \cite{MR3090230} to deduce that  $\n{F}= \R\varprojlim_s \n{F}/p^s$:  the case $\n{F}= \omega_{E,\infty}^k$ is clear as it is a line bundle.   Otherwise, we know that $\n{F}/p^s=\varinjlim_n \n{F}_n/p^s=\varinjlim_n (\n{F}_n/p^s\otimes_{X_n}\s{O}_{\f{X}_{\infty}})$ is a direct limit of $\s{O}_{\f{X}_{\infty}}/p^s$-line  bundles,  so that it is a quasi-coherent sheaf over $\f{X}_{\infty}$,  and the system  $\{\n{F}/p^s\}_{s\in \N}$  satisfies the Mittag-Leffler condition on  formal  affine perfectoids.   One obtains the quasi-isomorphism 
\[
\R\Gamma(\f{X}_{\infty}, \n{F})= \R\varprojlim_s \R\Gamma(\f{X}_{\infty},  \n{F}/p^s)
\]
whose cohomology translates into short exact sequences
\begin{equation}
\label{eqshortseqrlim}
0\rightarrow \R^1\varprojlim_s  \mathrm{H}^{i-1}(\f{X}_{\infty}, \n{F}/p^s) \rightarrow \mathrm{H}^{i}(\f{X}_{\infty}, \n{F}) \rightarrow \varprojlim_{s}  \mathrm{H}^{i}(\f{X}_{\infty}, \n{F}/p^s) \rightarrow 0.
\end{equation}
But part (2) implies that  $\mathrm{H}^i(\f{X}_{\infty},  \n{F}/p^s)= \varinjlim_{n} \mathrm{H}^i(X_n,  \n{F}_n/p^s)$ for all $s\in \bb{N}$.  As $X_n$ is a curve over $\n{O}_n$ and $\n{F}_n/p^s$ is supported in its special fiber, we know that $ \mathrm{H}^i(X_n,  \n{F}_n/p^s)=0$ for $i\geq 2$ and that the inverse system $\{\mathrm{H}^1(X_n, \n{F}_n/p^s)\}_{s\in \bb{N}}$  satisfies the ML condition.  This implies that $\mathrm{H}^i(\f{X}_{\infty}, \n{F}/p^s)=0$ for $i\geq 2$ and that the ML condition holds for  $\{\mathrm{H}^1(\f{X}_{\infty}, \n{F}/p^s)\}_{s\in \bb{N}}$.   From (\ref{eqshortseqrlim}) one obtains that $\mathrm{H}^i(\f{X}_{\infty}, \n{F})=0$ for $i\geq 2$.

We prove part (2).  Let $\f{U}=\{U_i\}_{i\in I}$ be a finite affine cover of $X$, let $\f{U}_n$  (resp. $\f{U}_{\infty}$) be its pullback to $X_n$ (resp. $\f{X}_{\infty}$).  As $\n{F}/p^s=\varinjlim_n \n{F}_n/p^s$  is a quasi-coherent $\s{O}_{\f{X}_{\infty}}/p^s$-module,     and the  (formal) schemes  $\f{X}_{\infty}$ and $X_n$  are separated,   we can use the \v{C}ech complex of $\f{U}_n$ (resp.  $\f{U}_{\infty}$) to compute the cohomology groups.  By definition we have 
\begin{eqnarray*}
\s{C}^\bullet(\f{U}_{\infty}, \n{F}/p^s) & = & \varinjlim_n \s{C}^\bullet(\f{U}_n, \n{F}_n/p^s), 
\end{eqnarray*}
then  (2) follows as filtered direct limits are exact.

The vanishing results of Proposition \ref{vanishingNeg} imply (3) for $k<0$ and $k>2$.  Let $k=1,2$ and   $p^{1/p}\in \n{O}^{cyc}$ be such that $|p^{1/p}|=|p|^{1/p}$.  As $\f{X}_{\infty}$ is integral perfectoid, the Frobenius $F:\f{X}_{\infty}/p\rightarrow \f{X}_{\infty}/p^{1/p}$ is an isomorphism.  Moreover,  $F^*(\omega_{E,\infty}^{k}/p^{1/p})=\omega_{E,\infty}^{pk}/p$ and $F^*(\omega_{E,\infty,cusp}^k/p^{1/p})=\omega_{E,\infty,cusp}^{pk}/p$ (notice that $F^*(D_n)=pD_n= D_{n-1}$).  Then,   Proposition \ref{vanishingNeg} (2) implies  
\begin{equation}
\label{vanishingperfint}
\mathrm{H}^1(\f{X}_{\infty}, \omega^k_{E,\infty}/p^{1/p})\cong \mathrm{H}^1(\f{X}_{\infty}, \omega^{pk}_{E,\infty}/p)=0,
\end{equation}
similarly for $\omega^{k}_{E,\infty,cusp}$.     By induction on $s$,  one shows that   $\mathrm{H}^1(\f{X}_{\infty},  \omega_{E,\infty}^k/p^s)=0$ and that $\mathrm{H}^0(\f{X}_{\infty}, \omega^k_{E,\infty}/p^{s+1})\rightarrow  \mathrm{H}^0(\f{X}_{\infty},  \omega^k_{E,\infty}/p^s)$ is surjective for all $s\in \bb{N}$ (resp. for $\omega_{E,\infty,cusp}^{k}$).  Taking derived inverse limits one gets $\mathrm{H}^1(\f{X}_{\infty}, \omega_{E,\infty}^k)=\mathrm{H}^1(\f{X}_{\infty}, \omega_{E,\infty,cusps}^k)=0$ and $\mathrm{H}^0(\f{X}_{\infty},  \omega^k_{E,\infty})= \varprojlim_s \mathrm{H}^0(\f{X}_{\infty},  \omega^k_{E,\infty}/p^s)$ (resp. for $\omega^k_{E,\infty, cusp}$).  This proves (3)  for $k=1,2$. 

Finally,  part (4) follows from part (2),   Proposition \ref{vanishingNeg} (3),   and the fact that  \[\mathrm{H}^0(\f{X}_{\infty}, \s{O}_{\f{X}_\infty})= \varprojlim_s \mathrm{H}^0(\f{X}_{\infty}, \s{O}_{\f{X}_\infty}/p^s)\]  by (\ref{eqshortseqrlim}) (resp. for $\s{O}_{\f{X}_{\infty}}(-D_{\infty})$).  

\endproof
\end{theo}

\begin{cor}
\label{CoroHiTorsionfree}
Let $\n{F}=\omega^k_{E,\infty}$ or $\omega^k_{E,\infty,cusp}$ for $k\in \bb{Z}$.  Then $\mathrm{H}^i(\f{X}_{\infty},  \n{F})/p^s= \mathrm{H}^i(\f{X}_{\infty}, \n{F}/p^s  )$ and $ \mathrm{H}^i(\f{X}_{\infty}, \n{F}  )=\varprojlim_s \mathrm{H}^i(\f{X}_{\infty}, \n{F} /p^s)$ for all $i,s\geq 0$.  In particular,  the cohomology groups $ \mathrm{H}^i(\f{X}_{\infty}, \n{F})$ are $p$-adically complete and separated.  Moreover, they are all torsion free. 
\proof
The case $k\neq 0$ follows since the cohomology complexes $\R\Gamma(\f{X}_{\infty},  \n{F}/p^s)$ are concentrated in only one degree, and $\R \Gamma(\f{X}_{\infty},\n{F})=\R \varprojlim_s \R\Gamma(\f{X}_{\infty}, \n{F}/p^s)$.  The case $k=0$ follows by part (4) of the previous theorem.  Namely,  $\mathrm{H}^0(\f{X}_{\infty},  \s{O}_{\f{X}_{\infty}}(-D_{\infty})/p^s)=0$ and $\mathrm{H}^0(\f{X}_{\infty},  \s{O}_{\f{X}_{\infty}}/p^s)=\n{O}^{cyc}/p^s$ for all $s\geq 0$.  Hence,  the inverse system of $\mathrm{H}^0$-cohomology groups satisfy the Mittag-Leffler condition,  and the $\R^1 \varprojlim$ appearing in the derived inverse  limit disappears for the $\mathrm{H}^1$-cohomology.    
\endproof
\end{cor}

As an application of the previous vanishing theorem,  we obtain vanishing results for the coherent  cohomology of the perfectoid modular curve.  Let  $(\n{X}_{\infty}, \s{O}_{\n{X}_{\infty}}^+)\rightarrow (\f{X}_{\infty}, \s{O}_{\s{O}_{X_{\infty}}})$ be the natural map of locally and topologically  ringed spaces provided by the generic fiber functor,  see Proposition  \ref{PropgenericFiberFunctor} and Theorem \ref{theoexistperfint}.  We define $\omega^{k,+}_{E,\eta}:= \omega_{E,\infty}^k \otimes_{\s{O}_{\f{X}_{\infty}}} \s{O}_{\n{X}_{\infty}}^+$ and $\omega^{k,+}_{E,cusp,\eta}:=\omega_{E,\infty,cusp}^k\widehat{\otimes}_{\s{O}_{\f{X}_{\infty}}} \s{O}_{\n{X}_{\infty}}^+$, where the completed tensor product is with respect to the $p$-adic topology.   As usual, we denote $\s{O}_{\n{X}_{\infty}}^+(-D_{\infty})= \omega_{E,cusp}^{0,+}$.   In the following we consider almost mathematics with respect to the maximal ideal of $\n{O}^{cyc}$. 

\begin{cor} 
\label{CoroCoherentPerfCurve}
The following holds 

\begin{enumerate}

\item The  cohomology complexes $\R\Gamma_{an}(\n{X}_{\infty},  \omega_{E,\eta}^{k,+})$ and $\R\Gamma_{an}(\n{X}_{\infty}, \omega_{E,cusp,\eta}^{k,+})$ of almost $\n{O}^{cyc}$-modules  are concentrated in degrees $[0,1]$ for all $k\in \bb{Z}$. 

\item  The sheaves $\omega_{E,\eta}^{k,+}$ and $\omega_{E,cusp,\eta}^{k,+}$ have cohomology   almost concentrated in degree $0$ for $k>0$.  Similarly,  the sheaves  $\omega_{E,\eta}^{k,+}$ and $\omega_{E,cusp,\eta}^{k,+}$  have cohomology  almost concentrated in degree $1$ for $k<0$.

\item $\mathrm{H}_{an}^0(\n{X}_{\infty}, \s{O}_{\n{X}_{\infty}}^+(-D_{\infty}))= 0$ and $\mathrm{H}_{an}^0(\n{X}_{\infty}, \s{O}_{\n{X}_{\infty}}^+)= \n{O}^{cyc} $.
\end{enumerate}

\proof
  We first prove the corollary   for $\n{F}=\omega_{E,\infty}^k$.  Let  $\n{F}^+_{\eta}$ denote the pullback of $\n{F}$ to $(\n{X}_{\infty}, \s{O}^+_{\n{X}_{\infty}})$.  Let $\f{U}=\{\n{U}_i\}_{i\in I}$ be an open cover of $\f{X}_{\infty}$ given by formal affine perfectoids arising from finite level such that $\omega_{E,\infty}|_{\n{U}_i}$ is trivial.   By Theorem \ref{theoexistperfint},  the generic fiber $ \n{U}_{i,\eta} $ of $\n{U}_i$  is an open affinoid perfectoid  subspace of $\n{X}_{\infty}$.  Let $\f{U}_{\eta}:= \{ \n{U}_{i,\eta}\}_{i\in I}$,  note that $\f{U}_{\eta}$ is a covering of $\n{X}_{\infty}$ and that the  restriction of $\n{F}^+_{\eta}$ to $\n{U}_{i,\eta}$ is trivial.  
By Scholze's  Almost Acyclicity Theorem for affinoid perfectoids,  $\n{F}^+_{\eta}|_{\n{U}_{i,\eta}}$ is almost acyclic for all $i\in I$.  The  \v{C}ech-to-derived functor spectral sequence gives us an almost quasi-isomorphism
\[
\s{C}^{\bullet}( \f{U}_{\eta},  \n{F}^+_{\eta} ) \simeq  \R\Gamma_{an}(\n{X}_{\infty}, \n{F}^+_{\eta}). \] 
On the other hand,  by the proof of Theorem  \ref{corVanishingCohPerfectoid} there is a  quasi-isomorphism 
\[
\s{C}^{\bullet}(\f{U},  \n{F} )\simeq \R\Gamma(\f{X}_{\infty}, \n{F}).
\]
But by definition of $\n{F}^+_{\eta}$,  and the fact that $\s{O}^+_{\n{X}_{\infty}}(\n{U}_{i,\eta})= \s{O}_{\f{X}_{\infty}}(\n{U}_i)$ by  Theorem \ref{theoexistperfint},   we actually have an almost equality $\s{C}^{\bullet}(\f{U}_{\eta}, \n{F}^+_{\eta}) =^{ae} \s{C}^{\bullet}(\f{U}, \n{F})$.   In other words, there is an almost quasi-isomorphism $  \R\Gamma_{an}(\n{X}_{\infty}, \n{F}^+_{\eta}) \simeq^{ae}  \R\Gamma(\f{X}_{\infty}, \n{F})$. 

Let $\n{F}_{cusp}=\omega_{E,\infty,cusp}^{k}$ and $\n{F}^+_{cusp, \eta}$ its pullback  to $(\n{X}_{\infty}, \s{O}^+_{\n{X}_{\infty}})$. To prove that \newline $ \R\Gamma_{an}(\n{X}_{\infty}, \n{F}^+_{cusp,\eta}) \simeq^{ae}  \R\Gamma(\f{X}_{\infty}, \n{F}_{cusp})$ we argue as follows:  note that we can write $\n{F}_{cusp}= \n{F}\otimes_{\s{O}_{\f{X}_{\infty}}} \s{O}_{\f{X}_{\infty}}(-D_{\infty})$. To apply the same argument as before we only need to show that $\s{O}_{\n{X}_{\infty}}^+(-D_{\infty})$ is almost acyclic over affinoid perfectoids of $\n{X}_{\infty}$.   Let $V(D_{\infty})\subset \n{X}_{\infty}$ be the perfectoid closed subspace defined by the cusps.  Note that $\s{O}_{\n{X}_{\infty}}(-D_\infty)$ is the ideal sheaf of $V(D_{\infty})$,   see the proof of  \cite[Theo.  IV.2.1]{scholzeTorsCoho2015} or  the explicit description of the completed stalks at the cusps of the integral  perfectoid modular curve.   Then,  we have an almost short exact sequence  for all $s\in \bb{N}$ 
\begin{equation}
\label{eqShortSequececusps}
0\rightarrow \s{O}^+_{\n{X}_{\infty}}(-D_{\infty})/p^s \rightarrow  \s{O}^+_{\n{X}_{\infty}}/p^s \rightarrow \s{O}^+_{V(D_{\infty})}/p^s \rightarrow 0.
\end{equation}
As the intersection of an affinoid perfectoid of $\n{X}_{\infty}$ with  $V(D_{\infty})$ is affinoid perfectoid,  and the second map of (\ref{eqShortSequececusps})   is surjective when evaluating at affinoid perfectoids of $\n{X}_{\infty}$,    Scholze's almost acyclicity  implies that $\s{O}^+_{\n{X}_{\infty}}(-D_{\infty})/p^s$ is almost acyclic in affinoid perfectoids.  Taking inverse limits and noticing that $\{\s{O}^+_{\n{X}_{\infty}}(-D_{\infty})/p^s \}_{s\in \bb{N}}$ satisfies the ML condition in affinoid perfectoids,  we get that $\s{O}^+_{\n{X}_{\infty}}(-D_{\infty})$  is almost acyclic in affinoid perfectoids of $\n{X}_{\infty}$.   The corollary follows from the vanishing results at the level of formal schemes.  
\endproof

\end{cor}

\begin{remark}
 As it was mentioned to me by Vincent Pilloni, the cohomological vanishing of the modular sheaves at infinite level provides many different exact sequences involving modular forms and the completed cohomology of the modular tower (to be defined in the next subsection).     Namely,    the primitive comparison theorem permits to compute the $\bb{C}_p$-scalar  extension of the completed cohomology as   $\mathrm{H}_{an}^1(\n{X}_{\infty,\bb{C}_p}, \s{O}_{\n{X}_\infty})$.  On the other hand,  the Hodge-Tate exact sequence  
 \[
 0\rightarrow \omega_E^{-1} \otimes_{\s{O}_{\n{X}}} \widehat{\s{O}}_{\n{X}} \rightarrow T_pE \otimes_{\widehat{\bb{Z}}_p} \widehat{\s{O}}_{\n{X}}  \rightarrow \omega_E\otimes_{\s{O}_{\n{X}}}\widehat{\s{O}}_{\n{X}}  \rightarrow 0
 \]
gives a short exact sequence over $\n{X}_{\infty}$
 \begin{equation}
 \label{eqHodgeTate}
  0\rightarrow \omega^{-1}_{E,\eta } \rightarrow \s{O}_{\n{X}_\infty,\bb{C}_p}^{\oplus 2} \rightarrow \omega_{E,\eta} \rightarrow 0
 \end{equation}
 via the universal trivialization of $T_pE$.   Then,  taking  the cohomology of  (\ref{eqHodgeTate}) one  obtains an exact sequence 
 \[
 0\rightarrow \bb{C}_p^{\oplus 2}\rightarrow \mathrm{H}_{an}^0(\n{X}_{\infty,\bb{C}_p}, \omega_{E,\eta}) \rightarrow \mathrm{H}_{an}^1(\n{X}_{\infty,\bb{C}_p}, \omega_{E,\eta}^{-1}) \rightarrow  \mathrm{H}_{an}^1(\n{X}_{\infty,\bb{C}_p} , \s{O}_{\n{X}_{\infty}})^{\oplus 2}\rightarrow 0.
 \]
 Another is example is given by tensoring (\ref{eqHodgeTate}) with $\omega_{E}$ and taking cohomology. One finds 
 \[
 0\rightarrow \bb{C}_p \rightarrow \mathrm{H}_{an}^0(\n{X}_{\infty,\bb{C}_p}, \omega_{E,\eta})^{\oplus 2}  \rightarrow \mathrm{H}_{an}^0(\n{X}_{\infty,\bb{C}_p}, \omega_{E,\eta}^{2}) \rightarrow \mathrm{H}_{an}^1(\n{X}_{\infty,\bb{C}_p} , \s{O}_{\n{X}_{\infty}})  \rightarrow 0.
 \]
It may be interesting a more careful study of these exact sequences.  
\end{remark}

\subsection*{Duality at infinite level}

Let $\n{F}=\omega^k_{E,\infty}$ or $\omega^{k}_{E,\infty,cusp}$ for $k\in \bb{Z}$,   let $\n{F}_n= \omega^k_{E,n}$ or $\omega^k_{E,n cusp}$ respectively.   Let $C$ be a non archimedean field extension of $K^{cyc}$ and $\n{O}_C$ its valuation ring.  Let $\f{X}_{\infty,C}$ be the extension of scalars of the integral modular curve to $\n{O}_C$.   Corollary \ref{CoroHiTorsionfree} says that  the cohomology groups $\mathrm{H}^i(\f{X}_{\infty}, \n{F})$ are torsion free,  $p$-adically complete and separated.  In particular,  we can endow $\mathrm{H}^i( \f{X}_{\infty, C}, \n{F})[\frac{1}{p}]$ with an structure of $C$-Banach space with unit ball $\mathrm{H}^i(\f{X}_{\infty,C}, \n{F})$.  The local duality theorem extends  to infinite level in the following way

\begin{theo}
\label{TheoDualityinfinite}
Let $\n{F}$ and $\n{F}_n$ be as above, and let $\n{F}_n^\vee= \s{H}om_{\s{O}_{X_n}}(\n{F}_n, \s{O}_{X_n})$ be the dual sheaf of $\n{F}_n$.   There is a $\GL_2(\bb{Q}_p)$-equivariant isomorphism of topological $\n{O}_C$-modules 
\begin{equation}
\label{eqInftyduality}
 \Hom_{\n{O}_C}( \mathrm{H}^i (\f{X}_{\infty,C},  \n{F}), \n{O}_C ) \cong \varprojlim_{n, \widetilde{\Tr}_n} \mathrm{H}^{1-i}(X_{n,\n{O}_C},  \n{F}_n^{\vee} \otimes \omega^2_{E,n,cusp}).
\end{equation}
The LHS is endowed with the  weak topology,  the RHS is endowed with the inverse limit topology,  $\widetilde{\Tr}_n$ are the normalized traces of Proposition \ref{isoDualizingModforms}, and the extension of scalars is given by $X_{n,\n{O}_C}=X_{n}\times_{\Spec  \n{O}_n} \Spec \n{O}_C$. 
\end{theo}

\begin{remark}
\label{RemarkDualityinfty}
\begin{enumerate}

\item   We could restate the previous theorem using $\omega^{\circ}_{n}= \pi_n^! \n{O}_n$ instead of $\omega^2_{E,n,cusp}$,   the trace $\widetilde{\Tr}_n$ would be replaced by the Serre duality trace relative to the morphism $X_{n+1, \n{O}_C}\rightarrow X_{n,\n{O}_C}$.  Note that even though the ring $\n{O}_C$ is not noetherian, all the objects involved are defined as pullbacks of objects which live over a finite extension of $\bb{Z}_p$, see Remark \ref{remark1}.   

\item Let $\n{F}^+_{\eta}= \n{F}\widehat{\otimes}_{\s{O}_{\f{X}_{\infty}}}  \s{O}^+_{\n{X}_{\infty}}$ be the pullback of $\n{F}$ to  $\n{X}_{\infty}$,  denote $\n{F}_{\eta}= \n{F}^+_{\eta}[\frac{1}{p}]$.  By Corollary \ref{CoroCoherentPerfCurve} we know that 
\[
\mathrm{H}^i(\n{X}_{\infty},  \n{F}_{\eta}) = \mathrm{H}^i(\f{X}_{\infty}, \n{F})[\frac{1}{p}].
\]
Thus, $\mathrm{H}^i(\n{X}_{\infty,C},  \n{F}_{\eta}) $ can be endowed with an structure of  $C$-Banach space.  Its dual is given by 
\[
\mathrm{H}^i(\n{X}_{\infty, C},  \n{F}_{\eta})^* = (\varprojlim_{m, \widetilde{\Tr}_n} \mathrm{H}^{1-i}(X_{n,\n{O}_C}, \n{F}_n^\vee \otimes \omega^2_{E,n,cusp}) )[\frac{1}{p}].
\]

\item Let $R_n:  \bb{Z}_p[\zeta_{N}]^{cyc}\rightarrow \bb{Z}_p[\zeta_{Np^n}]$ denote the $n$-th normalized Tate trace, and let $X'_n$ be the connected component of $X(N,p^n)_{\bb{Z}_p[\zeta_{N}]}$ corresponding to $\zeta_N$.   There is a natural  injective map 
\[
\varprojlim_{m , \widetilde{\Tr}_n} \mathrm{H}^{1-i}(X'_{n,\bb{Z}_p[\zeta_N]^{cyc}}, \n{F}_n^\vee \otimes \omega^2_{E,n,cusp}) \rightarrow \varprojlim_{n,  R_n\circ \widetilde{\Tr}_n}   \mathrm{H}^{1-i}(X'_{n}, \n{F}_n^\vee \otimes \omega^2_{E,n,cusp}).
\]
However,  this map is not surjective in general;  the RHS is profinite while the LHS is not compact.

\end{enumerate}

\end{remark}

Before proving Theorem \ref{TheoDualityinfinite} let us say some words   about the inverse limit of (\ref{eqInftyduality}),  it can be described as the kernel of the map 
\[
\prod_{n}  \mathrm{H}^{1-i}(X_{n,\n{O}_C},  \n{F}_n^{\vee} \otimes \omega^2_{E,n,cusp}) \xrightarrow{1- \widetilde{\Tr}_n} \prod_{n}  \mathrm{H}^{1-i}(X_{n,\n{O}_C},  \n{F}_n^{\vee} \otimes \omega^2_{E,n,cusp}).
\]
  Moreover,  the Corollary \ref{CoroHiTorsionfree} says that the factors in the products are $p$-adically complete,  separated and torsion free.   The following lemma implies that the inverse limit is always $p$-adically complete and separated 

\begin{lem}
Let $N$,  $M$ be torsion free,  $p$-adically complete and separated $\bb{Z}_p$-modules, and    $f:N \rightarrow M$  a $\bb{Z}_p$-linear  map.  Then $\ker f$ is torsion free,  $p$-adically complete and separated. 
\proof
It is clear that $\ker f$ is torsion free.   The map $f$ is continuous for the $p$-adic topology, in particular $\ker f \subset N$ is a closed sub-module.  Since $M$ is torsion free,  one has that $\ker f\cap  p^sN = p^s \ker f $ for all $s\geq 1$. Then, 
\[
\ker f = \varprojlim_s (\ker f/ (\ker f\cap p^s N))=  \varprojlim_s \ker f/p^s \ker f 
\]
proving the lemma.  \endproof
\end{lem}
 
Next, we recall the $\GL_2(\bb{Q}_p)$-action in both sides of (\ref{eqInftyduality}).  Without loss of generality we take $C=K^{cyc}$.   Let $ \chi:  \Gal(\n{O}^{cyc}/ \n{O}) \xrightarrow{\sim} \bb{Z}_p^{\times} $ be the cyclotomic character.  We define $\psi:  \GL_2(\bb{Q}_p)\rightarrow \Gal(\n{O}^{cyc}/ \n{O})$ to  be $\psi(g)= \chi^{-1} (p^{-v_p(\det g)} \det g)$, where $v_p:  \bb{Q}_p^*\rightarrow \bb{Z}$ denotes the $p$-adic valuation.  Fix $g\in \GL_2(\bb{Q}_p)$ and $n\geq 0$.  Let $m\geq 1$ be such that $\Gamma(p^m)\subset \Gamma(p^n)\cap g \Gamma(p^n)g^{-1}$,  write $c_g:  \GL_2(\bb{Q}_p)\rightarrow \GL_2( \bb{Q}_p)$ for the conjugation $x\mapsto g x g^{-1}$.  We denote by  $X(Np^n)_{c(g)}$ be the modular curve of level  $\Gamma(N)\cap \Gamma(p^n)\cap g \Gamma(p^n) g^{-1}$,  let $X_{n,c(g)}$ be the locus where the Weil pairing of the universal basis of $E[N]$ is equal to $\zeta_N\in \breve{\bb{Z}}_p$.      We let $\omega^{\circ}$ be the dualizing sheaf of $X_{n,c(g)}$,  i.e. the  exceptional inverse image of $\n{O}_{n,c(g)}:= \mathrm{H}^0(X_{n,c(g)}, \s{O}_{X_{n,c(g)}})$ over $X_{n,c_{g}}$.  

 The maps 
\[
\Gamma(p^m)\hookrightarrow \Gamma(p^n)\cap g\Gamma(p^n)g^{-1}\xrightarrow{c_{g^{-1}}} g^{-1}\Gamma(p^n)g \cap \Gamma(p^n)\hookrightarrow \Gamma(p^n)
\]
induce maps of modular curves 
\[
X_m \xrightarrow{q_1} X_{n,c(g)} \xrightarrow{g} X_{n,c(g^{-1})} \xrightarrow{q_2} X_{n},
\]
with $g$ an isomorphism.   Notice that  the modular sheaves $\omega^k_E$ are preserved by the pullbacks of   $q_1$, $q_2$ and $g$.   Let $\n{F}$ and $\n{F}_n$ be  as in Theorem \ref{TheoDualityinfinite},   we have induced maps of cohomology 
\[
\R\Gamma(X_n, \n{F}_n/p^s) \xrightarrow{   q_1^* \circ g^*\circ q_2^*} \R\Gamma(X_{m}, \n{F}_m/p^s).
\]
Taking direct limits we obtain a map
\[
\R\Gamma(\f{X}_{\infty},  \n{F}/p^s)\xrightarrow{g^*} \R\Gamma(\f{X}_{\infty},  \n{F}/p^s).
\]
Finally,  taking derived inverse limits one gets  the action of $g\in \GL_2(\bb{Q}_p)$ on the cohomology  $\R\Gamma(\f{X}_{\infty}, \n{F})$.   

The action of $\GL_2(\bb{Q}_p)$ on cohomology is not $\n{O}^{cyc}$-linear.  In fact,   it is $\psi$-semi-linear;  this can be shown by considering the Cartan decomposition \[\GL_2(\bb{Q}_p)= \bigsqcup_{\substack{n_1 \geq n_2 }}  \GL_2(\bb{Z}_p) \left( \begin{matrix} p^{n_1} & 0 \\ 0 & p^{n_2} \end{matrix} \right)  \GL_2(\bb{Z}_p) \] and using the compatibility of the Weil pairing with the determinant.    

The action of $\GL_2(\bb{Q}_p)$ on  $ \varprojlim_{n, \widetilde{\Tr}_n} \mathrm{H}^{1-i}(X_{n,\n{O}^{cyc}},  \n{F}_n^{\vee} \otimes \omega^2_{E,n,cusp})$ is defined in such a way that the isomorphism (\ref{eqInftyduality}) is equivariant.   Namely, there is a commutative diagram of local duality pairings provided by the functoriality of Serre duality 
\begin{equation}
\label{equationTracePullback}
\begin{tikzcd}[column sep=8pt]
\mathrm{H}^{1-i}(X_{m, \n{O}^{cyc}}, \n{F}_m^\vee\otimes \omega^2_{E,m,cusp} )\times \mathrm{H}^i(X_{m,\n{O}^{cyc}}, \n{F}_m \otimes  K/ \n{O} ) \ar[r]  \ar[d, shift right=1.5cm, "\widetilde{\Tr}_{q_1}"]& K^{cyc}/\n{O}^{cyc}  \ar[d, equal]\\
\mathrm{H}^{1-i}(X_{n,c(g), \n{O}^{cyc}},  (g\circ q_2)^* \n{F}_n^\vee \otimes \omega^\circ)\times \mathrm{H}^i(X_{n,c(g), \n{O}^{cyc}},  (g\circ q_2)^*\n{F}_n \otimes  K/ \n{O}  )\ar[r] \ar[d, shift right=1.5cm, "{g^{-1}}^*"]  \ar[u, shift right =1.5cm, "q_1^*"]& K^{cyc}/ \n{O}^{cyc} \\
\mathrm{H}^{1-i}(X_{n,c(g^{-1}), \n{O}^{cyc}},   q_2^*\n{F}_n^\vee \otimes \omega^\circ )\times \mathrm{H}^i(X_{n,c(g^{-1}), \n{O}^{cyc}},  q_2^* \n{F} \otimes  K/ \n{O}  ) \ar[r] \ar[d, shift right=1.5cm, "\widetilde{\Tr}_{q_2}"] \ar[u, shift right =1.5cm, "g^*"] & K^{cyc}/ \n{O}^{cyc}  \ar[d, equal] \ar[u, "\psi(g)"] \\
\mathrm{H}^{1-i}(X_{n,\n{O}^{cyc}},  \n{F}_n^\vee \otimes \omega^2_{E,n,cusp})\times \mathrm{H}^i(X_{n, \n{O}^{cyc}}, \n{F}_n\otimes K/ \n{O} ) \ar[r] \ar[u, shift right =1.5cm, "q_2^*"]& K^{cyc}/\n{O}^{cyc}.
\end{tikzcd}
\end{equation}
The  maps $\widetilde{\Tr}_{q_1}$ and $\widetilde{\Tr}_{q_2}$ are induced by the Serre duality traces of $q_1$ and $q_2$ respectively,  cf.  Remark \ref{RemarkDualityinfty} (1).  Thus, the right action of $g\in \GL_2(\bb{Q}_p)$ on a tuple $\underline{f}=(f_n)\in \varprojlim_{n, \widetilde{\Tr}_n} \mathrm{H}^{1-i}(X_{n,\n{O}^{cyc}},  \n{F}_n^\vee \otimes \omega^2_{E,n,cusp})$   is given by $\underline{f}|_{g}= ((\underline{f}|_g)_n )_{n\in \bb{N}}$, where  
\begin{eqnarray*}
(\underline{f}|_g)_n  & = &  \widetilde{\Tr}_{q_2}\circ {g^{-1}}^*\circ \widetilde{\Tr}_{q_1}(f_m)
\end{eqnarray*}
for $m$ big enough, and  $q_1$, $q_2$   as in  (\ref{equationTracePullback}). 
  
\begin{proof}[Proof of Theorem \ref{TheoDualityinfinite}]

Without loss of generality we take  $C=K^{cyc}$.  Let $\n{F}=\omega_{E,\infty}^k$ or $\omega^{k}_{E,\infty, cusp}$.   By Corollary \ref{CoroHiTorsionfree} we have  
\[
\mathrm{H}^i(\f{X}_{\infty}, \n{F})\otimes (K/ \n{O}) = \mathrm{H}^i( \f{X}_{\infty},  \n{F} \otimes  K/ \n{O} ). 
\]
Therefore
\begin{eqnarray*}
\Hom_{\n{O}^{cyc}}(\mathrm{H}^i(\f{X}_{\infty}, \n{F}),  \n{O}^{cyc}) & = & \Hom_{\n{O}^{cyc}}( \mathrm{H}^i(\f{X}_{\infty}, \n{F})\otimes K/ \n{O} , K^{cyc}/\n{O}^{cyc}) \\ 
 							 &  = &  \Hom_{\n{O}^{cyc}}( \mathrm{H}^i(\f{X}_{\infty}, \n{F}\otimes K/ \n{O} ) , K^{cyc}/\n{O}^{cyc}).
\end{eqnarray*}  
 On the other hand,  we have 
 \[
 \mathrm{H}^i( \f{X}_{\infty},  \n{F}\otimes  K/ \n{O} ) = \varinjlim_{n,p_n^*} \mathrm{H}^i(X_{n,\n{O}^{cyc}},  \n{F}_n\otimes   K/ \n{O} ) 
 \]
where the transition maps are given by pullbacks.  By local duality,  Proposition \ref{LocalDualityCurves}, we have a natural isomorphism 
\begin{eqnarray*}
\Hom_{\n{O}^{cyc}}(\mathrm{H}^i(\f{X}_{\infty}, \n{F}),  \n{O}^{cyc}) & = & \varprojlim_{n, p_n^*} \Hom_{\n{O}^{cyc}}(\mathrm{H}^i(X_{n,\n{O}^{cyc}}, \n{F}_n \otimes \bb{Q}_p / \bb{Z}_p),  K^{cyc}/ \n{O}^{cyc}) \\
 					& = & \varprojlim_{n,  \widetilde{\Tr}_n}  \mathrm{H}^{1-i}(X_{n,\n{O}^{cyc}}, \n{F}_n^{\vee}\otimes \omega^2_{E,n,cusp}).
\end{eqnarray*}
The isomorphism is $\GL_2(\bb{Q}_p)$-equivariant by the diagram (\ref{equationTracePullback}). 
\end{proof}

We end this section with an application of the local duality theorem at infinite level to the completed cohomology.  We let $\n{X}_{n,proet}$ be the pro-\'etale site of the finite level modular curve as in  \S 3 of  \cite{MR3090230},    and $\n{X}_{\infty,  proet}$  the pro-\'etale site  of the perfectoid modular curve as in  Lecture 8 of  \cite{ScholzeWeinspadicgeometry}. 

\begin{definition}
Let $i\geq 0$.  The \textit{$i$-th completed cohomology group} of the modular tower $\{\n{X}_n\}_{n\geq 0}$ is defined as 
\[
\widetilde{\mathrm{H}}^i:= \varprojlim_{s} \varinjlim_{n} \mathrm{H}_{et}^i( \n{X}_{n,\bb{C}_p}, \bb{Z}/p^s \bb{Z} ). 
\]
\end{definition}
\begin{remark}
The previous definition of completed cohomology is slightly different from the one of  \cite{MR2207783}.  Indeed, Emerton consider the \'etale cohomology with compact support of the affine modular curve $Y_n$. Let  $j: Y_n\rightarrow X_n$ be the inclusion and $\iota: D_n\rightarrow X_n$ be the cusp divisor,  both constructions are related by taking the  cohomology of the  short exact sequence 
\[
0 \rightarrow  j_!(\bb{Z}/p^s \bb{Z}) \rightarrow   \bb{Z}/p^s \bb{Z} \rightarrow \iota_* \iota^* \bb{Z}/p^s \bb{Z}\rightarrow 0. 
\]
Moreover, the cohomology at the cusps can be explicitely computed,  and many  interesting cohomology classes already appear in $\widetilde{\mathrm{H}}^1$. 
\end{remark}

We recall some important completed sheaves in the pro-\'etale site.  Let $\n{W}$ denote $\n{X}_n$ or $\n{X}_{\infty}$
\begin{itemize}
\item[•] We denote $\widehat{\bb{Z}}_p = \varprojlim_s \bb{Z}/p^s\bb{Z}$, the $p$-adic completion over $\n{W}_{proet}$ of the locally constant sheaf $\bb{Z}$.   

\item[•] Let $\widehat{\s{O}}^+_{\n{W}}= \varprojlim_{s} \s{O}_{\n{W}}^+/p^s$ be the $p$-adic completion of the structural sheaf of bounded functions over $\n{W}_{proet}$.

\end{itemize}
By Lemma 3.18 of \cite{MR3090230} the sheaf  $\widehat{\s{O}}_{\n{W}}^+$ is   the derived inverse limit of the projective system  $\{\s{O}_{\n{W}}^+/p^s\}_s$.  On the other hand,  the repleteness of the pro\'etale site and Proposition 3.1.10 of \cite{bhatt2014proetale} implies that $\widehat{\bb{Z}}_p$ is also the derived inverse limit of $\{\bb{Z}/p^s\bb{Z}\}_{s}$.   We have the following proposition 

\begin{prop}
\label{PropCompletedCohAndPerf}
Let $i\geq 0$,   there is a short exact sequence 
\[
0\rightarrow \R^1 \varprojlim_s \mathrm{H}^{i-1}_{et}(\n{X}_{\infty, \bb{C}_p},  \bb{Z}/p^s \bb{Z}) \rightarrow  \mathrm{H}^i_{proet}(\n{X}_{\infty,  \bb{C}_p}, \widehat{\bb{Z}}_p) \rightarrow \widetilde{\mathrm{H}}^i\rightarrow 0.
\]
\proof
As $\widehat{\bb{Z}}_p = \R \varprojlim_s \bb{Z}/p^s \bb{Z}$,  the Grothendieck spectral sequence for derived limits gives short exact sequences  for $i\geq 0$
\[
0\rightarrow \R^1 \varprojlim_s \mathrm{H}^{i-1}_{proet}(\n{X}_{\infty, \bb{C}_p},  \bb{Z}/p^s \bb{Z}) \rightarrow  \mathrm{H}^i_{proet}(\n{X}_{\infty,  \bb{C}_p}, \widehat{\bb{Z}}_p) \rightarrow   \varprojlim_s \mathrm{H}^{i}_{proet}(\n{X}_{\infty, \bb{C}_p},  \bb{Z}/p^s \bb{Z})  \rightarrow 0.
\]
Lemma 3.16 of \cite{MR3090230} implies that $\mathrm{H}^{i}_{proet}(\n{X}_{\infty, \bb{C}_p},  \bb{Z}/p^s \bb{Z}) =\mathrm{H}^{i}_{et}(\n{X}_{\infty, \bb{C}_p},  \bb{Z}/p^s \bb{Z}) $.   On the other hand,    Corollary 7.18 of \cite{MR3090258} gives an isomorphism
\[
\mathrm{H}_{et}^i(\n{X}_{\infty, \bb{C}_p}, \bb{Z}/p^s\bb{Z})= \varinjlim_{n} \mathrm{H}_{et}^i(\n{X}_{n,  \bb{C}_p}, \bb{Z}/p^s\bb{Z}),
\]
the proposition follows. 
\endproof
\end{prop}

Next, we relate the completed cohomologies $\widetilde{\mathrm{H}}^i$ with  the coherent cohomology of $\n{X}_{\infty}$ via the Primitive Comparison Theorem.  This strategy is the same  as the one  presented by Scholze in Chapter  IV of   \cite{scholzeTorsCoho2015}  for Emerton's completed cohomology.   In the following we work with the almost-setting with respect to the maximal ideal of $\n{O}_{\bb{C}_p}$

\begin{prop}[{\cite[Theo.  IV.2.1]{scholzeTorsCoho2015}}]
\label{PropCompletedCoherent}
There are natural almost isomorphisms
\begin{equation}
\label{eqIsopcoho}
\widetilde{\mathrm{H}}^i \widehat{\otimes }_{\bb{Z}_p} \n{O}_{\bb{C}_p} =^{ae} \mathrm{H}^i_{proet}(\n{X}_{\infty, \bb{C}_p} , \widehat{\s{O}}^+_{\n{X}_{\infty}}) =^{ae}  \mathrm{H}^i(\f{X}_{\infty}, \s{O}_{\f{X}_{\infty}}) \widehat{\otimes}_{\n{O}^{cyc}} \n{O}_{\bb{C}_p}.
\end{equation}
In particular,  $\widetilde{\mathrm{H}}^i=0$ for $i\geq 2$,   the $\R^1 \varprojlim_s$ of Proposition \ref{PropCompletedCohAndPerf} vanishes, and the $\widetilde{\mathrm{H}}^i$ are torsion free, $p$-adically complete and separated.  
\proof
By the Primitive Comparison Theorem \cite[Theo.  5.1]{MR3090230},  there are almost quasi-isomorphisms  for all $n,s,i\in \bb{N}$
\[ 
\mathrm{H}^i_{et}(\n{X}_{n,\bb{C}_p}, \bb{Z}/p^s \bb{Z})\otimes_{\bb{Z}_p} \n{O}_{\bb{C}_p} =^{ae} \mathrm{H}^i_{et}(\n{X}_{n,\bb{C}_p},  \s{O}^+_{\n{X}_n} /p^s ).
\]
Taking direct limits on $n$, and using Corollary 7.18 of  \cite{MR3090258}  one gets 
\begin{equation}
\label{eqderivedCompletedOhat1}
\mathrm{H}^i_{et}(\n{X}_{\infty, \bb{C}_p}, \bb{Z}/p^s \bb{Z}_p) \otimes_{\bb{Z}_p} \n{O}_{\bb{C}_p} =^{ae} \mathrm{H}^i_{et}(\n{X}_{\infty,  \bb{C}_p}, \s{O}_{\n{X}_{\infty}}^+/p^s).
\end{equation}
Namely,  we have $\s{O}^+_{\n{X}_{\infty}}/p^s= \varinjlim_{n} \s{O}^+_{\n{X}_n}/p^s$ as sheaves in the \'etale site of $\n{X}_{\infty}$. In fact,  let $U_{\infty}$ be  an affinoid perfectoid in the \'etale site of $\n{X}_{\infty}$ which factors as a composition of rational localizations and finite \'etale maps.   By Lemma 7.5 of \cite{MR3090258}  there exists $n_0\geq 0$ and an affinoid space $U_{n_0} \in \n{X}_{n,et}$ such that $U_{\infty}= \n{X}_{\infty}\times_{\n{X}_{n_0}} U_{n_0}$.   For $n\geq n_0$ denote  the pullback of $U_{n_0}$ to $\n{X}_{n,et}$ by $U_n$,  then $U_{\infty}\sim  \varprojlim_{n\geq n_0} U_{n}$ and $\s{O}^+(U_{\infty})/p^s=\varinjlim_{n\geq n_0} \s{O}^+(U_n)/p^s$.   

The sheaf $\s{O}^+_{\n{X}_{\infty}}/p^s$ is almost acyclic on affinoid perfectoids,  this implies that the RHS of (\ref{eqderivedCompletedOhat1}) is equal to  $\mathrm{H}^i_{an}(\n{X}_{\infty,\bb{C}_p}, \s{O}^+_{\n{X}_{\infty}}/p^s)$.  Then,   the proof of Corollary  \ref{CoroCoherentPerfCurve} allows us to compute the above complex using the formal model $\f{X}_{\infty}$
\begin{equation}
\label{eqcohintperft}
\mathrm{H}^i_{an}(\n{X}_{\infty, \bb{C}_p}, \s{O}^+_{\n{X}_{\infty}}/p^s) =^{ae} \mathrm{H}^i(\f{X}_{\infty}, \s{O}_{\f{X}_{\infty}}/p^s) \otimes_{\n{O}^{cyc}} \n{O}_{\bb{C}_p}. 
\end{equation}
The  Corollary \ref{CoroHiTorsionfree} shows that the inverse system $\{\mathrm{H}^i(\f{X}_{\infty},  \s{O}_{\f{X}_{\infty}}/p^s)\}_{s}$ satisfy the Mittag-Leffler condition.  As $\n{O}_{\bb{C}_p}/p^s$ is a faithfully flat $\bb{Z}/p^s\bb{Z}$-algebra,  the inverse system $\{\mathrm{H}^i_{et}(\n{X}_{\infty, \bb{C}_p}, \bb{Z}/p^s \bb{Z})\}_s$ also satisfies the Mittag-Leffler condition.  One  deduces from Proposition \ref{PropCompletedCohAndPerf} that 
\begin{equation}
\label{equationCompletedcoh}
\mathrm{H}^i_{proet}(\n{X}_{\infty, \bb{C}_p},  \widehat{\bb{Z}}_p)= \widetilde{\mathrm{H}}^i.
\end{equation}
We also obtain that $\widetilde{\mathrm{H}}^i/p^s= \mathrm{H}^i_{et}(\n{X}_{et,\bb{C}_p}, \bb{Z}/p^s \bb{Z})$ for all $i\in \bb{N}$.   Taking  inverse limits in (\ref{eqderivedCompletedOhat1}), and using (\ref{eqcohintperft}) and (\ref{equationCompletedcoh})  one obtains the corollary. 
\endproof
\end{prop}

We obtain a description of the dual of the completed cohomology in terms of cuspidal modular forms of weight $2$: 

\begin{theo}
There is a $\GL_2(\bb{Q}_p)$-equivariant isomorphism of almost $\n{O}_{\bb{C}_p}$-modules 
\[
\Hom_{\n{O}_{\bb{C}_p}} (\mathrm{\widetilde{H}}^1 \widehat{\otimes}_{\bb{Z}_p}\n{O}_{\bb{C}_p} ,  \n{O}_{\bb{C}_p} ) =^{ae} \varprojlim_{n, \widetilde{\Tr}_n} \mathrm{H}^0(X_{n, \n{O}_{\bb{C}_p}},  \omega^2_{E,n,cusp}). 
\]
\proof
This is  a  consequence of Proposition  \ref{PropCompletedCoherent}  and the particular case of  Theorem \ref{TheoDualityinfinite} when  $\n{F}= \s{O}_{\f{X}_{\infty}}$ and $C=\bb{C}_p$.
\endproof
\end{theo}

\addcontentsline{toc}{section}{References}

\bibliographystyle{alpha}
\bibliography{intmodperf2}

\begin{center}
•
\end{center}

\end{document}